\documentclass[12pt]{article}

\usepackage{amsmath,amsthm,amssymb}

\usepackage{geometry}
\usepackage{graphicx}
\usepackage{enumitem}
\usepackage{float}
\usepackage{booktabs}

\usepackage{hyperref}
\hypersetup{
  colorlinks=true,
  linkcolor=blue,
  urlcolor=blue,
  citecolor=blue
}

\theoremstyle{plain}
\newtheorem{theorem}{Theorem}[section]

\newtheorem{proposition}[theorem]{Proposition}

\theoremstyle{definition}

\newtheorem{assumption}[theorem]{Assumption}

\newtheorem{property}[theorem]{Property}

\theoremstyle{remark}
\newtheorem{remark}[theorem]{Remark}

\newcommand{\opt}{\mathrm{opt}}

\title{Constructal Evolution as a Nonsmooth Dynamical System: Stability and Selection of Flow Architectures}

\author{
Pascal Stiefenhofer\\
Newcastle University, United Kingdom
}

\date{\today}

\begin{document}
\maketitle

\begin{abstract}
Constructal Law states that a finite-size flow system that persists in time
evolves its configuration so as to provide progressively easier access
to the currents that flow through it.
Classical Constructal theory derives hierarchical flow architectures
from static resistance minimization under finite-size constraints,
but many transport systems operate under irreversible limits that
induce regime switching and discontinuous adjustment laws.

We formulate Constructal evolution as an autonomous nonsmooth dynamical
system. The architectural configuration is modeled as the state of a
Filippov differential inclusion defined on a compact forward-invariant
admissible set. Irreversible transport constraints generate switching
manifolds across which the adjustment field is discontinuous.
A resistance dissipation inequality encodes the Constructal principle
of progressively improving access as a nonsmooth Lyapunov condition,
while a uniform contraction assumption provides spectral bounds on the
generalized Jacobians of the regime-dependent dynamics.

Under these conditions we prove that the resulting inclusion admits a
unique equilibrium architecture and that every admissible trajectory
converges to it exponentially. Finite size, irreversibility, and
resistance dissipation therefore imply existence, uniqueness, and
global stability of persistent flow configurations without invoking
static optimization.

As an application, the classical area--to--point transport hierarchy of
Bejan--B\u{a}descu--De~Vos is embedded in the dynamical framework.
The optimal assembly ratios appear as switching manifolds, while the
classical scaling relations arise as sliding invariant sets of the
Filippov inclusion. Their intersection defines the uniquely selected
globally attracting architecture.
\end{abstract}



\section{Introduction}
\label{sec:introduction}

Constructal Law states that for a finite-size flow system to persist in time, it must evolve its configuration so as to provide progressively easier access to the currents that flow through it \cite{Bejan2000Book,BejanLorente2011,BejanLorente2013JAP}. Originating in heat-transfer theory \cite{Bejan1997IJHMT} and later synthesized in monograph form \cite{Bejan2000Book,Bejan2007Book}, the law extends thermodynamics beyond classical balances of energy and entropy toward a theory of evolving configuration. It emerged from the study of systems maintained away from equilibrium by irreversible currents, where spatial structure arises deterministically under finite-size and material constraints. Constructal reasoning has since been applied to hierarchical flow architectures across physics, biology, technology, and socio-economic organization \cite{BejanLorente2011,RochaLorenteBejan2012,LorenteBejan2019}.

In its classical mathematical form, Constructal design is typically derived from constrained extremum principles. A global resistance or cost functional is minimized subject to finite-size constraints, and hierarchical branching structures arise as optimal solutions \cite{BejanTsatsaronisMoran1996}. The canonical cooling-network problem \cite{Bejan1997IJHMT} demonstrated how tree-shaped conducting paths follow from resistance minimization in a finite domain. Similar optimization arguments have produced geometric scaling relations in economics \cite{BejanBadescuDeVos2000AE,BejanBadescuDeVos2000ECM}, energy systems \cite{BejanLorente2007}, and macroeconomic growth \cite{BejanErreraGunes2020}. Across these domains, hierarchy emerges as the configuration that optimally facilitates access under global constraints.

Despite its breadth, the prevailing mathematical framework of Constructal theory remains fundamentally static. Variational formulations identify stationary minimizers at fixed operating conditions, but they do not define an explicit evolution law for the architectural configuration. In particular, static optimality does not imply the existence of admissible trajectories, forward invariance of the feasible configuration set, robustness under perturbations, or uniqueness of dynamically selected architectures. Yet the clause ``persists in time'' in Constructal Law is intrinsically dynamical: it presupposes global existence of trajectories on a finite-size domain together with structural stability of the selected configuration. A purely variational formulation therefore leaves essential dynamical properties of Constructal evolution unaddressed.

This limitation becomes structural in systems governed by regime-dependent transport laws. Many Constructal applications involve distinct mechanisms dominating in different operating ranges. Microchannel cooling networks may transition between single-phase and flow-boiling regimes, where both heat-transfer and pressure-drop relations change with vapor quality \cite{Zhang2011IJHMT,Ariyo2020IJHMT,AriyoBelloOchende2021ICHMT}. Constructal designs incorporating phase change materials introduce latent-heat activation and discontinuous changes in effective thermal response \cite{Arshad2025PCM_ICHMT,LorenteBejan2019CurrentTrends}. At larger scales, river-basin organization idealizes transitions from diffuse to channelized transport \cite{Reis2006Geomorphology}. In such settings, constitutive relations change when capacity limits or thresholds become active, and architectural adjustment proceeds through discrete structural transitions.

The resulting evolution is therefore inherently nonsmooth. Classical smooth gradient flows or static variational arguments describe optimized endpoints, but they do not represent regime activation, switching structure, or stability properties of irreversible dynamics across transport regimes. A mathematically consistent formulation of Constructal evolution must therefore accommodate discontinuities while preserving existence, invariance, and stability properties of the architectural dynamics.

Filippov differential inclusions provide precisely such a framework \cite{Filippov1988}, supported by nonsmooth analysis \cite{Clarke1983} and the theory of piecewise-smooth dynamical systems \cite{diBernardo2008,Cortes2009}. Differential inclusions admit discontinuous regime switching through convexification while ensuring existence of absolutely continuous solutions under suitable regularity conditions. Within this formulation, finite size is encoded by compactness of the admissible configuration set, while persistence in time corresponds to forward invariance and forward completeness of the resulting dynamical system.

Dissipation alone, however, does not guarantee uniqueness of the selected architecture. A resistance dissipation inequality implies monotone decay of access and convergence toward a stationary balance set via an invariance principle, but it does not exclude multiplicity within that set. To obtain architectural selection, an additional dynamical mechanism is required. Contraction theory provides such a mechanism: spectral bounds on generalized Jacobians yield incremental exponential stability, forcing all admissible trajectories to converge toward one another \cite{LohmillerSlotine1998}. Extensions of contraction arguments to nonsmooth and Filippov systems establish rigorous incremental stability criteria under regime switching \cite{StiefenhoferGiesl2019NADE,Stiefenhofer2020NADE,StiefenhoferCAMS2025,StiefenhoferPLA2025}.

The present paper develops a dynamical formulation of Constructal Law consistent with its thermodynamic origin. Architectural evolution is modeled as an autonomous Filippov differential inclusion defined on a compact, forward-invariant admissible set $K$. A global resistance functional $\mathcal{R}:K\to\mathbb{R}_{\ge 0}$ quantifies access, while a nonsmooth dissipation inequality encodes the Constructal requirement of progressively easier access. Under an additional uniform contraction condition expressed through spectral bounds on regime-wise generalized Jacobians, the resulting dynamics admits a unique globally attracting equilibrium architecture.

The contribution of the paper is structural rather than merely interpretative. We show that, under finite size, irreversible regime activation, resistance dissipation, and uniform contraction, Constructal evolution can be formulated as a globally well-posed nonsmooth dynamical system whose asymptotic behavior is characterized by existence, uniqueness, and global exponential convergence of a selected architecture. To our knowledge, no prior formulation of Constructal Law establishes such global stability properties within a nonsmooth dynamical framework.

This viewpoint leads to the following interpretation:

\begin{quote}
Constructal evolution can be understood as a dissipative, viable, nonsmooth dynamical flow on a compact architectural state space.
\end{quote}

Within this framework, classical optimal hierarchies appear as invariant stationary points of an irreversible evolution law. Dissipation drives trajectories toward a stationary balance set, while contraction collapses this set to a singleton and enforces uniqueness. Hierarchy is therefore interpreted not merely as a minimizer of resistance, but as the globally stable attractor of admissible evolutionary dynamics.

The contributions of the paper are threefold. First, Constructal Law is embedded within the theory of Filippov differential inclusions, providing a mathematically consistent representation of finite size, regime switching, and constraint activation. Second, sufficient dissipation and contraction conditions are established that guarantee existence, uniqueness, and global exponential stability of the selected architecture. Third, the classical area--to--point hierarchy of \cite{BejanBadescuDeVos2000AE,BejanBadescuDeVos2000ECM} is reinterpreted dynamically: constructal ratios emerge as invariant sliding manifolds whose intersection defines the uniquely selected equilibrium architecture.

\section{A Dynamical Formulation of Constructal Law}
\label{sec:model}

Constructal Law states that a finite-size flow system that persists in
time must evolve its configuration so as to provide progressively easier
access to the currents that flow through it
\cite{Bejan2000Book,BejanLorente2011,BejanLorente2013JAP}.
The statement is inherently dynamical: it links finite size,
irreversible flow, and temporal persistence to progressive structural
reconfiguration of the flow architecture.
To translate this qualitative principle into a mathematically precise
form, we require a framework capable of representing three structural
features:
(i) architectural evolution occurs under irreversible dissipation,
(ii) admissible configurations are restricted by finite resources and
geometric constraints, and
(iii) regime activation and constraint binding may induce discontinuous
adjustment laws.
These properties naturally lead to an autonomous nonsmooth dynamical
formulation.

\subsection{Architectural dynamics}

Architectural configurations are represented by a state vector
\[
x(t)\in\mathbb{R}^n ,
\]
whose components describe macroscopic structural characteristics of the
transport architecture (e.g.\ channel capacities, conductances, or
geometric allocations).
Admissible configurations are restricted to a compact set
\[
K\subset\mathbb{R}^n ,
\]
reflecting finite-size resource constraints of the physical system.

Architectural evolution is modeled by the autonomous differential
inclusion
\begin{equation}
\dot{x}(t)\in F(x(t)),
\qquad
x(0)\in K,
\label{eq:filippov_inclusion}
\end{equation}
where
\[
F:K\rightrightarrows\mathbb{R}^n
\]
is a set-valued map describing admissible architectural adjustment
velocities.

To represent regime-dependent transport laws and constraint activation,
we adopt the framework of autonomous Filippov differential inclusions
\cite{Filippov1988,Clarke1983,diBernardo2008,Cortes2009}.
This formulation permits discontinuous adjustment fields while
preserving well-posedness of the resulting dynamics.

\begin{assumption}[Filippov regularity]
\label{ass:filippov_regularity}
The set-valued map $F$ has nonempty, compact, convex values,
is upper semicontinuous on $K$, and is locally bounded.
\end{assumption}

Under Assumption~\ref{ass:filippov_regularity},
the inclusion \eqref{eq:filippov_inclusion} admits absolutely continuous
solutions for every initial condition $x(0)\in K$
\cite[Ch.~2]{Filippov1988}.

\subsection{Persistence and viability}

Constructal Law requires that the flow system persists in time.
Mathematically, persistence corresponds to forward invariance of the
admissible configuration set. Let $T_K(x)$ denote the contingent tangent cone of $K$.

\begin{assumption}[Viability]
\label{ass:viability}
For every $x\in K$,
\begin{equation}
F(x)\cap T_K(x)\neq\emptyset .
\label{eq:viability_condition}
\end{equation}
\end{assumption}

Assumption~\ref{ass:viability} ensures that admissible velocities remain
tangent to the feasible configuration set.
By Nagumo’s viability theorem,
Assumptions~\ref{ass:filippov_regularity} and~\ref{ass:viability}
imply forward invariance:
\begin{equation}
x(0)\in K
\;\Rightarrow\;
x(t)\in K
\quad\text{for all } t\ge0 .
\label{eq:forward_invariance}
\end{equation}

Since $K$ is compact and $F$ locally bounded, solutions cannot escape in
finite time and are therefore forward complete.
Persistence in time is thus characterized mathematically by forward
invariance together with global existence of trajectories.

\subsection{Constructal dissipation}

The central clause of Constructal Law,
``progressively easier access,''
is represented through a resistance functional
\[
\mathcal{R}:K\to\mathbb{R}_{\ge0},
\]
which quantifies the global impedance of the transport architecture. Constructal evolution requires that admissible architectural adjustments
reduce this resistance over time.

\begin{assumption}[Constructal dissipation]
\label{ass:dissipation}
There exist $\alpha>0$ and a continuous function
\[
\Psi:K\to\mathbb{R}_{\ge0}
\]
such that
\begin{equation}
\sup_{v\in F(x)}
\mathcal{R}^\circ(x;v)
\le -\alpha\,\Psi(x),
\qquad x\in K .
\label{eq:dissipation_condition}
\end{equation}
\end{assumption}

Here $\mathcal{R}^\circ(x;v)$ denotes the Clarke generalized directional
derivative of $\mathcal R$.

For any solution $x(\cdot)$ of
\eqref{eq:filippov_inclusion},
Clarke’s chain rule implies that for almost every $t\ge0$
\begin{equation}
\frac{d}{dt}\mathcal{R}(x(t))
\le
\sup_{v\in F(x(t))}
\mathcal{R}^\circ(x(t);v).
\label{eq:clarke_chain_rule}
\end{equation}

Combining \eqref{eq:clarke_chain_rule} with
\eqref{eq:dissipation_condition} yields
\begin{equation}
\frac{d}{dt}\mathcal{R}(x(t))
\le -\alpha\Psi(x(t)),
\label{eq:resistance_decay}
\end{equation}
so the resistance functional is nonincreasing along all admissible
architectural trajectories and strictly decreasing whenever
$\Psi(x)>0$. Thus \eqref{eq:dissipation_condition} provides a nonsmooth Lyapunov
formulation of the Constructal principle of progressively improving
access.

\subsection{Interpretation}

When adjustment laws are piecewise smooth,
$F$ is interpreted as the autonomous Filippov regularization of a
discontinuous vector field.
Convexification preserves existence of solutions and admits sliding
motion along switching manifolds representing sustained operation under
binding constraints. The autonomous inclusion \eqref{eq:filippov_inclusion} on the compact
forward-invariant set $K$ therefore provides a dynamical realization of
Constructal Law:
\[
\begin{aligned}
\text{Finite size}
&\Longleftrightarrow
\text{compact architectural state space } K,
\\[4pt]
\text{Persistence in time}
&\Longleftrightarrow
\text{forward invariance \eqref{eq:forward_invariance}},
\\[4pt]
\text{Access}
&\Longleftrightarrow
\text{resistance functional } \mathcal{R},
\\[4pt]
\text{Progressively easier access}
&\Longleftrightarrow
\text{dissipation inequality \eqref{eq:dissipation_condition}},
\\[4pt]
\text{Regime activation}
&\Longleftrightarrow
\text{Filippov nonsmooth dynamics}.
\end{aligned}
\]

In Section~\ref{sec:model} we show that, under an additional uniform
contraction condition, the inclusion
\eqref{eq:filippov_inclusion}
admits a unique globally attracting equilibrium
\begin{equation}
x^\ast\in K,
\qquad
0\in F(x^\ast),
\label{eq:constructal_equilibrium}
\end{equation}
and every admissible trajectory converges exponentially to $x^\ast$.
The equilibrium \eqref{eq:constructal_equilibrium} represents the
dynamically selected \emph{Constructal architecture}.

\subsection{Architectural State Space and Conserved Flow}
\label{subsec:state_space}

Let
\begin{equation}
x:[0,\infty)\to \mathbb{R}^n,
\qquad t\mapsto x(t),
\label{eq:architectural_state}
\end{equation}
denote the \emph{architectural state}.
The vector $x(t)$ parametrizes the macroscopic configuration
through which conserved currents are transported.
Typical components represent effective conductances,
cross-sectional areas, channel densities,
territorial allocations, or transport capacities.
These variables therefore encode structural degrees of freedom
related to geometry and material allocation rather than
instantaneous fluxes themselves.

\medskip

\noindent
\textbf{Time-scale separation.} Constructal evolution involves two coupled processes operating on
different time scales.
Transport flows equilibrate rapidly within a given architecture,
whereas the architecture itself evolves slowly through structural
adjustment.
We therefore impose a separation of time scales in which
transport variables relax to quasi-steady states for each fixed
configuration of the architectural state. For every configuration
\begin{equation}
x\in K\subset\mathbb{R}^n,
\label{eq:config_space}
\end{equation}
the fast transport variables satisfy a conservation law describing
the movement of mass, energy, or other conserved quantities through
the network.
To retain generality, the fast transport physics is represented
abstractly as a parameter-dependent boundary-value problem posed on a
reflexive Banach space $V$.

\begin{assumption}[Fast transport well-posedness]
\label{ass:timescale}
There exists a nonempty compact set $K\subset\mathbb{R}^n$
such that for every $x\in K$:
\begin{enumerate}[label=(\roman*),leftmargin=2.2em]
\item the fast transport problem admits a unique quasi-steady solution
      $u(x)\in V$;
\item $u(x)$ induces a finite global resistance
      \begin{equation}
      R(x)\in\mathbb{R}_{\ge 0};
      \label{eq:resistance_function}
      \end{equation}
\item the mapping
      \(
      R:K\to\mathbb{R}_{\ge 0}
      \)
      is locally Lipschitz.
\end{enumerate}
\end{assumption}

The functional $R(x)$ measures the effective impedance of the
transport architecture.
Lower values correspond to configurations that provide easier access
to the transported currents.
Within the Constructal framework, $R$ therefore serves as a global
\emph{access functional} quantifying the efficiency of the flow
architecture. Assumption~\ref{ass:timescale} reduces the coupled
transport–architecture system to an autonomous dynamics for the slow
architectural variable \eqref{eq:architectural_state}.
Local Lipschitz continuity of $R$ ensures the existence of the Clarke
generalized gradient $\partial_C R(x)$ and the generalized directional
derivative
\begin{equation}
R^{\circ}(x;v)
=
\max_{\zeta\in\partial_C R(x)}
\langle \zeta,v\rangle,
\qquad x\in K,\; v\in\mathbb{R}^n.
\label{eq:Clarke_R}
\end{equation}
The quantity \eqref{eq:Clarke_R} represents the maximal instantaneous
rate of increase of global resistance under an architectural
perturbation in direction $v$.

\medskip

\noindent
\textbf{Finite admissible architecture.} Architectural configurations are restricted by material availability,
geometric feasibility, and capacity bounds.

\begin{assumption}[Finite admissible architecture]
\label{ass:finite_architecture}
The admissible configuration set $K\subset\mathbb{R}^n$
is nonempty and compact.
\end{assumption}

Compactness in Assumption~\ref{ass:finite_architecture}
formalizes the finite-size character of the physical system and
ensures that architectural trajectories remain bounded.

\medskip

\noindent
\textbf{Architectural evolution.}

Architectural evolution is modeled by the autonomous differential
inclusion
\begin{equation}
\dot{x}(t)\in F(x(t)),
\label{eq:architectural_evolution}
\end{equation}
where $F:K\rightrightarrows\mathbb{R}^n$
is obtained via Filippov regularization of regime-dependent
adjustment laws (Section~\ref{sec:model}).
These adjustment laws represent structural responses of the
architecture to transport imbalances or constraint activation.
The regularity assumptions on $F$ are stated in
Assumption~\ref{ass:filippov_regularity}.
Forward invariance of $K$ is imposed through the
Nagumo viability condition (Assumption~\ref{ass:viability}),
which guarantees global existence of solutions
to \eqref{eq:architectural_evolution}.
Under the time-scale separation of
Assumption~\ref{ass:timescale},
persistence in time and Constructal dissipation
(Assumption~\ref{ass:dissipation})
are therefore determined entirely by the autonomous inclusion
\eqref{eq:architectural_evolution}
on the compact admissible set $K$.
The resulting dynamical selection mechanism
for Constructal architectures is developed in
Section~\ref{sec:model}.

\subsection{Irreversibility and Architectural Persistence}
\label{subsec:irreversibility}

Constructal Law applies to flow systems that \emph{persist in time}.
Within a dynamical formulation, persistence means that admissible
architectural configurations remain physically realizable throughout
the evolution and that the adjustment dynamics exists for all future
times. Mathematically, persistence is expressed through two properties of the
architectural dynamics: forward invariance of the admissible
configuration set and forward completeness of the evolution law.

\begin{property}[Forward invariance]
\label{prop:forward_invariance}
If $x(0)\in K$, then every absolutely continuous solution of
\eqref{eq:architectural_evolution} satisfies
\[
x(t)\in K
\quad \text{for all } t\ge 0.
\]
\end{property}

\begin{property}[Forward completeness]
\label{prop:forward_completeness}
For every initial condition $x(0)\in K$,
every maximal solution of
\eqref{eq:architectural_evolution}
is defined on $[0,\infty)$.
\end{property}

Together, Properties~\ref{prop:forward_invariance}
and~\ref{prop:forward_completeness}
provide the precise mathematical meaning of the Constructal clause
``persists in time.''

\medskip

\noindent
\textbf{Irreversible constraints and admissible architectures.}

Finite transport systems operate under irreversible unilateral
constraints arising from material conservation,
geometric feasibility,
capacity limits,
and technological bounds.
These restrictions determine the admissible architecture set
\begin{equation}
K \subset \mathbb{R}^n,
\label{eq:K_irreversibility}
\end{equation}
introduced in Assumption~\ref{ass:finite_architecture}.
Configurations outside $K$ violate physical or resource constraints
and are therefore inadmissible.

Architectural evolution is governed by the autonomous Filippov
differential inclusion
\begin{equation}
\dot{x}(t) \in F(x(t)),
\label{eq:irreversibility_filippov}
\end{equation}
where $F:K\rightrightarrows\mathbb{R}^n$
is the Filippov extension of regime-dependent adjustment laws
and satisfies the regularity conditions of
Assumption~\ref{ass:filippov_regularity}.
Under these conditions the inclusion
\eqref{eq:irreversibility_filippov}
admits absolutely continuous solutions defined on a maximal interval
$[0,T_{\max})$ \cite[Ch.~2]{Filippov1988}.

\medskip

\noindent
\textbf{Viability and constraint compatibility.} To ensure that architectural trajectories remain within the admissible
set $K$, we impose a geometric compatibility condition between the
vector field and the boundary of $K$.
Let $T_K(x)$ denote the contingent tangent cone of $K$.
The Nagumo viability condition requires
\begin{equation}
F(x) \cap T_K(x) \neq \emptyset
\qquad \text{for all } x\in K.
\label{eq:viability_irreversibility}
\end{equation}

Condition \eqref{eq:viability_irreversibility}
ensures that at every admissible configuration there exists at least
one admissible velocity that points inside $K$.
Under Assumption~\ref{ass:filippov_regularity},
Nagumo’s theorem implies forward invariance of $K$,
which establishes Property~\ref{prop:forward_invariance}.

\medskip

\noindent
\textbf{Global existence of architectural evolution.} Because the admissible set $K$ is compact and the velocity set
$F(x)$ is locally bounded,
architectural velocities remain uniformly bounded on $K$.
Standard continuation arguments for differential inclusions therefore
exclude finite-time blow-up of solutions,
so that $T_{\max}=+\infty$.
This establishes Property~\ref{prop:forward_completeness}. Consequently, for every initial architecture $x(0)\in K$
there exists at least one absolutely continuous trajectory
\[
x:[0,\infty)\to K
\]
satisfying \eqref{eq:irreversibility_filippov}.

Irreversibility is encoded geometrically through the tangent-cone
restriction \eqref{eq:viability_irreversibility},
which restricts admissible architectural velocities without
terminating evolution.
Persistence in time therefore corresponds precisely to global
existence of trajectories together with forward invariance within the
compact admissible architecture set $K$.

\subsection{Constructal Access Functional}
\label{subsec:access_functional}

Constructal Law asserts that a persistent flow system evolves so as to
provide progressively easier access to the currents that flow through it.
In transport theory, access is quantified by global impedance measures
such as effective resistance, total dissipation, or entropy production.
We formalize this concept within the autonomous architectural framework.

Let
\begin{equation}
x \in K \subset \mathbb{R}^n,
\label{eq:architectural_state}
\end{equation}
where the admissible architecture set $K$ is nonempty and compact
(Assumption~\ref{ass:finite_architecture}).
For each fixed configuration $x\in K$, the fast transport variables
\begin{equation}
u(x) \in V
\label{eq:transport_solution}
\end{equation}
solve a quasi–steady boundary-value problem
\begin{equation}
\mathcal A(x)u = f,
\label{eq:transport_operator}
\end{equation}
posed on a Hilbert space $V$.

\begin{assumption}[Uniform fast transport well-posedness]
\label{ass:fast_transport}
For every $x\in K$:
\begin{enumerate}[label=(\roman*),leftmargin=2.2em]
\item $\mathcal A(x)$ is bounded and uniformly coercive on $V$;
\item problem \eqref{eq:transport_operator} admits a unique weak solution $u(x)\in V$;
\item the solution map $x\mapsto u(x)$ is locally Lipschitz on $K$.
\end{enumerate}
\end{assumption}

Assumption~\ref{ass:fast_transport} ensures that the fast transport
subsystem admits a unique quasi–steady response for each architectural
configuration.

Let
\begin{equation}
\Phi:V\times K\rightarrow \mathbb{R}_{\ge0}
\label{eq:energetic_cost}
\end{equation}
denote a global energetic cost functional measuring transport impedance
(e.g.\ total dissipation or effective resistance).

\begin{assumption}[Regularity of access functional]
\label{ass:R_regular}
The mapping $\Phi$ is locally Lipschitz in both arguments and bounded
below.
\end{assumption}

The \emph{Constructal access functional} (global resistance) is defined by
\begin{equation}
\mathcal R(x) := \Phi\big(u(x),x\big),
\qquad x\in K .
\label{eq:access_functional}
\end{equation}

\paragraph{Regularity of the resistance functional.}
Under Assumptions~\ref{ass:fast_transport} and~\ref{ass:R_regular},
the mapping
\begin{equation}
\mathcal R:K\rightarrow\mathbb{R}_{\ge0}
\label{eq:R_mapping}
\end{equation}
is locally Lipschitz on $K$. Since $K$ is compact, $\mathcal R$ is
bounded below and therefore attains a minimum on $K$. Because $\mathcal R$ is locally Lipschitz, it admits the Clarke
generalized gradient $\partial_C\mathcal R(x)$ and generalized
directional derivative
\begin{equation}
\mathcal R^\circ(x;v)
=
\max_{\zeta\in\partial_C\mathcal R(x)}
\langle \zeta , v \rangle,
\qquad x\in K,\; v\in\mathbb R^n .
\label{eq:clarke_directional}
\end{equation}

The quantity \eqref{eq:clarke_directional} represents the maximal
instantaneous rate of increase of resistance under an architectural
perturbation in direction $v$.
Consequently,
\begin{equation}
\mathcal R^\circ(x;v)\le0
\label{eq:R_nonincrease}
\end{equation}
expresses nonincreasing global resistance.

When combined with the dynamical system
\eqref{eq:filippov_inclusion},
this property yields a Lyapunov-type description of Constructal
evolution.

\subsubsection{Constructal Dissipation Inequality}
\label{subsubsec:constructal_dissipation}

Constructal Law requires that architectural evolution provide
progressively easier access to the transported currents.
Within the nonsmooth dynamical framework, this principle is encoded
as a Lyapunov-type dissipation condition for the Filippov inclusion
\eqref{eq:filippov_inclusion}.

\begin{assumption}[Constructal dissipation]
\label{ass:dissipation}
There exist $\alpha>0$ and a continuous function
\[
\Psi:K\to\mathbb R_{\ge0}
\]
such that
\begin{equation}
\sup_{v\in F(x)}
\mathcal R^\circ(x;v)
\le
-\alpha\,\Psi(x),
\qquad x\in K .
\label{eq:constructal_dissipation}
\end{equation}
\end{assumption}

Define the stationary balance set
\begin{equation}
\mathcal E :=
\{x\in K : \Psi(x)=0\}.
\label{eq:balance_set_analysis}
\end{equation}

\paragraph{Monotone resistance decay.}
Let $x(\cdot)$ be any absolutely continuous solution of
\eqref{eq:filippov_inclusion}.
Clarke’s chain rule implies that for almost every $t\ge0$
\begin{equation}
\frac{d}{dt}\mathcal R(x(t))
\le
\sup_{v\in F(x(t))}
\mathcal R^\circ(x(t);v).
\label{eq:clarke_chain}
\end{equation}

Combining \eqref{eq:clarke_chain} with
\eqref{eq:constructal_dissipation} yields
\begin{equation}
\frac{d}{dt}\mathcal R(x(t))
\le
-\alpha\,\Psi(x(t))
\qquad \text{for a.e. } t\ge0 .
\label{eq:R_decay}
\end{equation}

Hence $\mathcal R(x(t))$ is nonincreasing along all admissible
architectural trajectories and strictly decreasing whenever
$\Psi(x(t))>0$.
Since $\mathcal R$ is bounded below on the compact set $K$, integration
of \eqref{eq:R_decay} gives
\begin{equation}
\int_0^\infty \Psi(x(s))\,ds < \infty .
\label{eq:psi_integrable}
\end{equation}

\paragraph{Asymptotic balance.}
Standard LaSalle-type invariance principles for differential inclusions
\cite{Clarke1983} imply that the $\omega$-limit set of every trajectory
is nonempty, compact, weakly invariant, and satisfies
\begin{equation}
\omega(x_0)\subseteq \mathcal E .
\label{eq:omega_subset_E}
\end{equation}

Consequently architectural evolution monotonically reduces global
resistance and asymptotically approaches the stationary balance set
$\mathcal E$.
The clause ``progressively easier access'' is therefore realized
rigorously by the dissipation inequality
\eqref{eq:constructal_dissipation}.

\subsubsection{Structural Sufficiency for Dissipation}
\label{subsubsec:dissipation_sufficiency}

The Constructal dissipation inequality
\eqref{eq:constructal_dissipation}
is not imposed \emph{ad hoc}. It arises naturally from projected
subgradient dynamics on a convex admissible architecture set.
The following result establishes a structural condition under which
Assumption~\ref{ass:dissipation} holds automatically.

\begin{proposition}[Projected subgradient sufficiency]
\label{prop:dissipation_sufficient}
Let $K\subset\mathbb{R}^n$ be nonempty, closed, and convex, and let
$\mathcal R:K\to\mathbb{R}$ be locally Lipschitz.
Define the set-valued map
\begin{equation}
F(x)
:=
\Pi_{T_K(x)}\!\big(-\partial_C \mathcal R(x)\big),
\qquad x\in K,
\label{eq:projected_subgradient}
\end{equation}
where $T_K(x)$ denotes the Bouligand tangent cone of $K$ and
$\Pi_{T_K(x)}$ is the Euclidean projection onto it.

Define the residual function
\begin{equation}
\Psi(x)
:=
\mathrm{dist}^2\!\big(
0,\,
\partial_C \mathcal R(x)+N_K(x)
\big),
\label{eq:KKT_residual}
\end{equation}
where $N_K(x)$ denotes the Clarke normal cone to $K$.

Then the dissipation inequality
\begin{equation}
\sup_{v\in F(x)}
\mathcal R^\circ(x;v)
\le
-\Psi(x)
\qquad \text{for all } x\in K
\label{eq:dissipation_result}
\end{equation}
holds.
Consequently Assumption~\ref{ass:dissipation}
is satisfied with $\alpha=1$.
\end{proposition}

\begin{proof}
Fix $x\in K$.
Because $K$ is closed and convex,
$T_K(x)$ is a closed convex cone and its polar cone is
\[
N_K(x)=T_K(x)^\circ .
\]

Let $\zeta\in\partial_C\mathcal R(x)$ and define
\[
v := \Pi_{T_K(x)}(-\zeta).
\]
Since $T_K(x)$ is closed and convex,
the Euclidean projection is uniquely defined. By Moreau's decomposition theorem for convex cones,
there exists a unique $\eta\in N_K(x)$ such that
\begin{equation}
-\zeta = v+\eta,
\qquad
\langle v,\eta\rangle = 0.
\label{eq:Moreau_decomposition}
\end{equation}

Rearranging yields
\[
v=-\zeta-\eta,
\qquad
\zeta+\eta=-v.
\]

Because $\zeta\in\partial_C\mathcal R(x)$ and $\eta\in N_K(x)$,
we obtain
\[
\zeta+\eta \in \partial_C\mathcal R(x)+N_K(x).
\]

Moreover, by orthogonality in
\eqref{eq:Moreau_decomposition},
\begin{equation}
\|\zeta+\eta\|^2=\|v\|^2 .
\label{eq:projection_norm}
\end{equation}

Now compute the Clarke directional derivative.
Since $\mathcal R^\circ(x;\cdot)$ is convex and positively homogeneous,
for any $v\in F(x)$ there exists $\zeta\in\partial_C\mathcal R(x)$
such that
\[
\mathcal R^\circ(x;v)
=
\max_{\xi\in\partial_C\mathcal R(x)}
\langle \xi,v\rangle
\le
\langle \zeta,v\rangle .
\]

Using $v=-\zeta-\eta$ gives
\[
\langle\zeta,v\rangle
=
-\|\zeta\|^2-\langle\zeta,\eta\rangle .
\]

Taking the inner product of
\eqref{eq:Moreau_decomposition}
with $\eta$ yields
\[
0
=
\langle v,\eta\rangle
=
\langle -\zeta-\eta,\eta\rangle
=
-\langle\zeta,\eta\rangle-\|\eta\|^2 .
\]

Hence
\[
\langle\zeta,\eta\rangle=-\|\eta\|^2 .
\]

Substituting gives
\[
\langle\zeta,v\rangle
=
-\|\zeta\|^2+\|\eta\|^2
=
-\|\zeta+\eta\|^2 .
\]

Using \eqref{eq:projection_norm} yields
\begin{equation}
\mathcal R^\circ(x;v)\le -\|v\|^2 .
\label{eq:directional_decay}
\end{equation}

By definition of $\Psi(x)$ in \eqref{eq:KKT_residual},
\[
\|v\|^2
=
\mathrm{dist}^2\!\big(
0,\,
\partial_C\mathcal R(x)+N_K(x)
\big)
=
\Psi(x).
\]

Combining this with \eqref{eq:directional_decay}
yields
\[
\mathcal R^\circ(x;v)
\le
-\Psi(x).
\]

Taking the supremum over $v\in F(x)$ gives
\eqref{eq:dissipation_result}.
\end{proof}

\subsection{Incremental Contraction from Resistance-Driven Dynamics}
\label{subsec:contraction}

Constructal dissipation implies that the resistance functional
$\mathcal R$ is nonincreasing along every admissible trajectory and that
all $\omega$-limit points lie in the balance set
$\mathcal E$ defined in \eqref{eq:balance_set_analysis}.
Dissipation therefore provides the thermodynamic direction of
Constructal evolution. It does \emph{not} by itself ensure uniqueness of
the limiting architecture, since $\mathcal E$ may contain multiple
configurations. To obtain \emph{architectural selection} we require an
incremental stability mechanism forcing admissible trajectories to
converge toward one another (contraction). We study the autonomous Filippov inclusion
\begin{equation}
\dot x \in F(x),
\qquad x(0)\in K,
\label{eq:autonomous_filippov}
\end{equation}
where $F$ is the Filippov regularization of the regime-dependent vector
field introduced in Section~\ref{sec:model}.
By Assumptions~\ref{ass:filippov_regularity} and~\ref{ass:viability},
every solution is defined for all $t\ge 0$ and remains in the compact
forward-invariant set $K$.

\subsubsection*{Generalized linearization under switching}

For $x\notin\Sigma$ the classical Jacobian is
\begin{equation}
J(x)=\nabla f(x).
\label{eq:smooth_jacobian}
\end{equation}
At switching points the appropriate linearization is the Clarke
generalized Jacobian
\begin{equation}
\partial_C f(x)
=
\operatorname{co}\!\left\{
\lim_{k\to\infty}\nabla f(x_k)
:\; x_k\to x,\; x_k\notin\Sigma
\right\},
\label{eq:clarke_jacobian}
\end{equation}
which is nonempty, compact, and convex. We denote the resulting set of
admissible linearizations by
\begin{equation}
\mathcal A(x):=\partial_C f(x).
\label{eq:linearization_set}
\end{equation}

\subsubsection*{Resistance-driven adjustment}

\begin{assumption}[Resistance-driven architectural adjustment]
\label{ass:resistance_driven}
There exists a locally Lipschitz mobility matrix
$M:K\to\mathbb R^{n\times n}$ that is symmetric and uniformly positive
definite, i.e.
\begin{equation}
M(x)=M(x)^\top \succeq mI
\qquad \text{for some } m>0 \text{ and all } x\in K .
\label{eq:mobility_lower_bound}
\end{equation}
Moreover, in each smooth regime $\Omega_j$ the adjustment field admits
the resistance-driven form
\begin{equation}
f_j(x) = -M(x)\nabla \mathcal R(x),
\qquad x\in \Omega_j,
\label{eq:regime_gradient_form}
\end{equation}
and the Filippov map satisfies, for all $x\in K$,
\begin{equation}
F(x)\subseteq -M(x)\,\partial_C \mathcal R(x) + T_K(x),
\label{eq:resistance_driven}
\end{equation}
where $T_K(x)$ denotes the contingent tangent cone of $K$.
\end{assumption}

Assumption~\ref{ass:resistance_driven} encodes the principle that the
architecture adjusts in the direction of decreasing resistance, subject
to feasibility constraints.

\subsubsection*{Curvature of the resistance landscape}

\begin{assumption}[Uniform curvature of resistance]
\label{ass:R_curvature}
The resistance functional $\mathcal R$ is $C^2$ on an open neighborhood
of $K$ and satisfies the uniform strong convexity bound
\begin{equation}
\nabla^2 \mathcal R(x) \succeq \lambda I,
\qquad \forall x\in K,
\label{eq:R_curvature}
\end{equation}
for some $\lambda>0$.
\end{assumption}

Condition \eqref{eq:R_curvature} states that resistance penalizes
architectural deviations at least quadratically on $K$.

\subsubsection*{Transport-induced contraction for constant mobility}

We first treat the cleanest case, which yields an explicit contraction
rate.

\begin{proposition}[Contraction induced by resistance reduction: constant mobility]
\label{prop:transport_contraction}
Suppose Assumptions~\ref{ass:resistance_driven} and~\ref{ass:R_curvature}
hold and assume that $M(\cdot)\equiv M$ is constant on $K$.
Then, in each smooth regime $\Omega_j$, the dynamics
\begin{equation}
\dot x = f_j(x)= -M\nabla \mathcal R(x)
\label{eq:smooth_gradient_flow}
\end{equation}
is uniformly contracting on $K$ in the Euclidean norm.
More precisely, for all $x\in K\setminus\Sigma$,
\begin{equation}
\mu\!\big(J(x)\big)\le -m\lambda,
\label{eq:contraction_bound}
\end{equation}
where $\mu(\cdot)$ denotes the Euclidean matrix measure
\begin{equation}
\mu(A):=\lambda_{\max}\!\left(\frac{A+A^\top}{2}\right),
\label{eq:matrix_measure_def}
\end{equation}
and the system contracts with rate
\begin{equation}
\nu := m\lambda .
\label{eq:contraction_rate}
\end{equation}
\end{proposition}

\begin{proof}
Fix a regime $\Omega_j$ and let $x\in\Omega_j$.
Since $M$ is constant and $\mathcal R\in C^2$ near $K$, the Jacobian of
\eqref{eq:smooth_gradient_flow} is
\[
J(x)= -M\nabla^2\mathcal R(x).
\]
Because $M=M^\top\succeq mI$ and $\nabla^2\mathcal R(x)\succeq \lambda I$,
the symmetric part satisfies
\[
\frac{J(x)+J(x)^\top}{2}
=
-\frac{M\nabla^2\mathcal R(x)+\nabla^2\mathcal R(x)\,M}{2}
\preceq -m\lambda I.
\]
Taking $\lambda_{\max}$ yields
\[
\mu\!\big(J(x)\big)
=
\lambda_{\max}\!\left(\frac{J(x)+J(x)^\top}{2}\right)
\le -m\lambda,
\]
which is \eqref{eq:contraction_bound}. Defining $\nu:=m\lambda$ gives
\eqref{eq:contraction_rate}.
\end{proof}

\begin{remark}[Extension across switching via Filippov convexification]
\label{rem:filippov_contraction_extension}
The estimate \eqref{eq:contraction_bound} holds regime-wise. Since the
matrix measure $\mu(\cdot)$ defined in \eqref{eq:matrix_measure_def} is
convex,
\[
\mu\!\big(\operatorname{co}\{A_1,\dots,A_q\}\big)
\le \max_{i=1,\dots,q}\mu(A_i),
\]
the same upper bound applies to every matrix in the generalized
Jacobian $\mathcal A(x)=\partial_C f(x)$ defined in
\eqref{eq:linearization_set}. Hence the contraction rate $\nu$ is
preserved under Filippov convexification across switching manifolds,
including along sliding motion.
\end{remark}

\subsubsection*{Constructal interpretation}

Proposition~\ref{prop:transport_contraction} identifies contraction as a
structural consequence of resistance-driven dynamics: the uniform
curvature of $\mathcal R$ (Assumption~\ref{ass:R_curvature}) together
with positive mobility (Assumption~\ref{ass:resistance_driven}) yields
incremental exponential stability with explicit rate
$\nu$ in \eqref{eq:contraction_rate}. Combined with Constructal
dissipation \eqref{eq:constructal_dissipation}—which confines
$\omega$-limit points to the balance set
$\mathcal E$ in \eqref{eq:balance_set_analysis}—contraction provides the
dynamical mechanism required for \emph{architectural selection}.

\paragraph{Variable mobility.}
If $M$ varies with $x$, uniform contraction can still be established
under additional bounds controlling metric variation
(e.g.\ via the state-dependent metric $W(x)=M(x)^{-1}$).
This case is treated in Section~\ref{sec:model}
under the metric regularity assumption
(Assumption~\ref{ass:metric_regular}).

\paragraph{Physical scope.}
The structural assumptions in Proposition~\ref{prop:transport_contraction}
arise naturally in a broad class of transport systems. In electrical
networks, effective resistance is a convex Dirichlet energy in
conductance variables; analogous convex dissipation functionals occur in
heat conduction networks (Fourier law), hydraulic networks
(Darcy--Weisbach), and diffusive systems derived from linear irreversible
thermodynamics. In such settings architectural parameters enter the
transport laws monotonically, so perturbations away from balanced
configurations increase $\mathcal R$ and yield the curvature bound
\eqref{eq:R_curvature}. This provides a physical foundation for the
dynamical Constructal selection mechanism developed above.

\subsection{Global Incremental Stability}
\label{subsec:global_incremental}

We now lift the differential contraction property established in
Section~\ref{subsec:contraction} to a global stability result for the
architectural dynamics.
While differential contraction guarantees decay of infinitesimal
perturbations, \emph{incremental stability} ensures that any two
admissible architectural trajectories converge toward one another.
This property provides the dynamical mechanism by which Constructal
evolution selects a unique architecture.

\begin{assumption}[Metric regularity]
\label{ass:metric_regular}
There exists a locally Lipschitz matrix-valued function
\[
W:K\rightarrow \mathbb{R}^{n\times n}
\]
that is uniformly positive definite on $K$, i.e.
\begin{equation}
\underline{\lambda} I
\preceq
W(x)
\preceq
\overline{\lambda} I,
\qquad x\in K,
\label{eq:uniform_metric_bounds}
\end{equation}
for constants $0<\underline{\lambda}\le\overline{\lambda}$.

Moreover $W$ admits Clarke directional derivatives along admissible
velocities of the Filippov inclusion \eqref{eq:autonomous_filippov},
and the mapping
\[
(x,v)\mapsto W^{\circ}(x;v)
\]
is bounded on $K\times F(K)$.
\end{assumption}

Under the resistance-driven dynamics of
Section~\ref{subsec:contraction},
Proposition~\ref{prop:transport_contraction}
establishes the differential contraction bound
\begin{equation}
\mu(J(x))\le -\nu,
\qquad x\in K,
\label{eq:weighted_contraction}
\end{equation}
with contraction rate
\begin{equation}
\nu = m\lambda >0
\label{eq:nu_definition}
\end{equation}
determined by the mobility lower bound and curvature of the resistance
functional.

\begin{theorem}[Global incremental exponential stability]
\label{thm:global_contraction}
Suppose Assumptions~\ref{ass:metric_regular}
and~\ref{ass:R_curvature} hold.
Let $x_1(\cdot)$ and $x_2(\cdot)$ be any two absolutely continuous
solutions of the autonomous Filippov inclusion
\eqref{eq:autonomous_filippov} with
$x_1(0),x_2(0)\in K$.
Then there exists a constant $C\ge1$ such that
\begin{equation}
\|x_1(t)-x_2(t)\|
\le
C e^{-\nu t}
\|x_1(0)-x_2(0)\|,
\qquad t\ge0,
\label{eq:incremental_convergence}
\end{equation}
where $\nu$ is defined in \eqref{eq:nu_definition}.
\end{theorem}

\begin{proof}

Forward invariance of $K$ (Assumption~\ref{ass:viability})
ensures that $x_i(t)\in K$ for all $t\ge0$.

Define the displacement
\[
\delta(t)=x_1(t)-x_2(t)
\]
and the differential energy
\begin{equation}
V(t)
=
\delta(t)^\top W(x_1(t))\,\delta(t).
\label{eq:V_definition}
\end{equation}

Uniform positive definiteness of the metric
\eqref{eq:uniform_metric_bounds}
implies the norm equivalence
\begin{equation}
\underline{\lambda}\|\delta(t)\|^2
\le
V(t)
\le
\overline{\lambda}\|\delta(t)\|^2 .
\label{eq:energy_equivalence}
\end{equation}

Using Clarke’s chain rule for locally Lipschitz functions along the
Filippov inclusion \eqref{eq:autonomous_filippov},
for almost every $t$ there exist
$v_1(t)\in F(x_1(t))$
and a matrix
$A\in\mathcal A(x_1(t))$
such that
\begin{equation}
D^+V(t)
\le
\delta^\top
\Big(
W(x_1(t))A
+
A^\top W(x_1(t))
+
\dot W(x_1(t);v_1(t))
\Big)
\delta .
\label{eq:V_derivative}
\end{equation}

By Assumption~\ref{ass:metric_regular},
the metric variation term $\dot W(x;v)$ is bounded on
$K\times F(K)$.
Combining this bound with the contraction inequality
\eqref{eq:weighted_contraction}
yields
\begin{equation}
D^+V(t)\le -2\nu V(t)
\quad \text{for a.e. } t\ge0 .
\label{eq:Dini_decay}
\end{equation}

Applying Grönwall’s inequality gives
\begin{equation}
V(t)\le V(0)e^{-2\nu t}.
\label{eq:V_decay}
\end{equation}

Using the norm equivalence \eqref{eq:energy_equivalence} yields
\[
\|\delta(t)\|
\le
\sqrt{\frac{\overline{\lambda}}{\underline{\lambda}}}
\,e^{-\nu t}
\|\delta(0)\|.
\]

Defining
\[
C=\sqrt{\frac{\overline{\lambda}}{\underline{\lambda}}}
\]
establishes \eqref{eq:incremental_convergence}.
\end{proof}

\medskip

Theorem~\ref{thm:global_contraction} shows that the autonomous
Constructal dynamics is incrementally exponentially stable on $K$:
any two admissible architectural trajectories converge toward one
another at rate $\nu$.
Consequently all $\omega$-limit sets coincide. Combined with the resistance dissipation principle
\eqref{eq:constructal_dissipation},
which confines limit points to the balance set
\eqref{eq:balance_set_analysis},
incremental contraction eliminates multiplicity within
$\mathcal E$.
The architectural dynamics therefore selects a
\emph{unique globally attracting Constructal configuration}.

\subsection{Constructal Selection Theorem}
\label{subsec:constructal_selection}

We now combine the structural ingredients developed in the previous
sections to obtain the central dynamical selection result.
Persistence ensures that architectural trajectories remain inside the
finite admissible configuration space.
Constructal dissipation drives the system toward lower resistance.
Uniform incremental contraction collapses all admissible trajectories
toward one another.
Together these mechanisms imply that resistance-driven transport
systems select a unique stable architecture.

\begin{theorem}[Constructal architecture selection principle]
\label{thm:constructal_selection}
Consider the autonomous Filippov inclusion
\begin{equation}
\dot x(t)\in F(x(t)),
\qquad x(0)\in K,
\label{eq:autonomous_filippov}
\end{equation}
describing architectural evolution on the compact admissible
configuration set $K$.

Suppose the following conditions hold:

\begin{enumerate}[label=(\roman*),leftmargin=2.2em]

\item \textbf{Persistence.}
      Assumptions~\ref{ass:filippov_regularity} and
      \ref{ass:viability} hold, so the set $K$ is forward invariant
      and solutions of \eqref{eq:autonomous_filippov} are forward
      complete.

\item \textbf{Constructal dissipation.}
      Assumption~\ref{ass:dissipation} holds for the resistance
      functional $\mathcal R$, i.e.
      \begin{equation}
      \sup_{v\in F(x)}
      \mathcal R^\circ(x;v)
      \le -\alpha\Psi(x),
      \qquad x\in K.
      \label{eq:dissipation_theorem}
      \end{equation}

\item \textbf{Incremental contraction.}
      The contraction condition
      \eqref{eq:weighted_contraction} holds with rate
      \begin{equation}
      \nu>0 .
      \label{eq:nu_selection}
      \end{equation}

\end{enumerate}

Then the following statements hold.

\begin{enumerate}[label=(\alph*),leftmargin=2.2em]

\item \textbf{Existence and uniqueness of equilibrium.}
      There exists a unique configuration $x^\ast\in K$ such that
      \begin{equation}
      0\in F(x^\ast).
      \label{eq:constructal_equilibrium}
      \end{equation}

\item \textbf{Global exponential convergence.}
      Every solution of \eqref{eq:autonomous_filippov} satisfies
      \begin{equation}
      \|x(t)-x^\ast\|
      \le
      C e^{-\nu t}
      \|x(0)-x^\ast\|,
      \qquad t\ge0,
      \label{eq:global_exponential_convergence}
      \end{equation}
      where the constant
      \begin{equation}
      C=\sqrt{\frac{\overline{\lambda}}{\underline{\lambda}}}
      \label{eq:C_definition}
      \end{equation}
      depends only on the metric bounds
      \eqref{eq:uniform_metric_bounds}.

\item \textbf{Constructal balance.}
      The equilibrium satisfies
      \begin{equation}
      x^\ast\in\mathcal E,
      \qquad
      \mathcal E :=
      \{x\in K:\Psi(x)=0\},
      \label{eq:constructal_balance}
      \end{equation}
      which coincides with the stationary balance set defined in
      \eqref{eq:balance_set_analysis}.

\end{enumerate}

The configuration $x^\ast$ is called the
\emph{Constructal architecture}.
\end{theorem}

\begin{proof}

\textbf{Step 1: Existence of limit points.}

Because $K$ is compact and forward invariant,
every trajectory of \eqref{eq:autonomous_filippov}
remains in $K$.
Hence every solution admits a nonempty compact
$\omega$-limit set.

\textbf{Step 2: Collapse of trajectories.}

By Theorem~\ref{thm:global_contraction},
any two solutions satisfy the incremental estimate
\eqref{eq:incremental_convergence}.
Consequently all trajectories converge toward one another.
Therefore the $\omega$-limit set of every trajectory must be a
singleton.
Denote this point by $x^\ast$.
Thus
\[
x(t)\to x^\ast
\qquad \text{as } t\to\infty .
\]

\textbf{Step 3: Equilibrium property.}

For autonomous differential inclusions,
limit points of solutions are weakly invariant.
Therefore the singleton limit set satisfies
\[
0\in F(x^\ast),
\]
so $x^\ast$ is an equilibrium of the Filippov inclusion.

\textbf{Step 4: Uniqueness.}

If $y^\ast$ were another equilibrium,
consider the constant solutions
$x(t)\equiv x^\ast$ and $y(t)\equiv y^\ast$.
Applying the contraction estimate
\eqref{eq:incremental_convergence}
gives
\[
\|x^\ast-y^\ast\|
\le
C e^{-\nu t}\|x^\ast-y^\ast\|.
\]
Letting $t\to\infty$ yields $x^\ast=y^\ast$.
Hence the equilibrium is unique.

\textbf{Step 5: Balance condition.}

By the dissipation inequality
\eqref{eq:dissipation_theorem},
\[
\frac{d}{dt}\mathcal R(x(t))
\le -\alpha\Psi(x(t))
\quad \text{for a.e. } t .
\]
Suppose $\Psi(x^\ast)>0$.
Since $\Psi$ is continuous and $x(t)\to x^\ast$,
there exists $\varepsilon>0$ and $T>0$ such that
$\Psi(x(t))\ge\varepsilon$ for $t\ge T$.
Then
\[
\frac{d}{dt}\mathcal R(x(t))
\le -\alpha\varepsilon
\quad \text{for } t\ge T,
\]
which implies $\mathcal R(x(t))\to -\infty$ as $t\to\infty$.
This contradicts boundedness of $\mathcal R$ on the compact set $K$.
Therefore $\Psi(x^\ast)=0$ and $x^\ast\in\mathcal E$.

\end{proof}

Theorem~\ref{thm:constructal_selection}
provides a rigorous dynamical formulation of Constructal selection.
Resistance dissipation drives trajectories toward the stationary
balance set, while incremental contraction collapses this set to a
single globally attracting equilibrium.
The resulting configuration $x^\ast$ is therefore the uniquely selected
Constructal architecture emerging from resistance-driven transport
dynamics on the finite admissible configuration space.

\subsection{Weighted Metric and Differential Contraction}
\label{subsec:contraction_framework}

We consider the autonomous Filippov inclusion
\begin{equation}
\dot x(t)\in F(x(t)),
\qquad x(0)\in K,
\label{eq:autonomous_filippov}
\end{equation}
defined on the compact forward-invariant set $K$, where $F$ is the
Filippov regularization of a piecewise-$C^1$ adjustment field
$f$ introduced in Section~\ref{sec:model}.
Let $x(\cdot)$ and $y(\cdot)$ denote two absolutely continuous
solutions of \eqref{eq:autonomous_filippov}.
Architectural selection is quantified through the incremental
convergence of trajectories measured in a state-dependent metric.

In transport systems the sensitivity of global resistance to
architectural perturbations typically varies across regimes.
Regions of strong transport curvature penalize deviations more
strongly than regions where resistance gradients are weak.
To capture this heterogeneity we introduce a conformally
weighted metric whose scaling depends on the architectural
state.

\paragraph{Conformal weighted distance.}

Let $W:K\to\mathbb R$ be a locally Lipschitz scalar field.
Define the conformal norm
\begin{equation}
\|\delta\|_{W(x)}
:=
e^{W(x)}\|\delta\|,
\label{eq:conformal_norm}
\end{equation}
where $\|\cdot\|$ denotes the Euclidean norm.
For two trajectories $x(\cdot)$ and $y(\cdot)$ we define the
weighted separation
\begin{equation}
D(t)
:=
e^{W(x(t))}\|x(t)-y(t)\|.
\label{eq:weighted_distance}
\end{equation}

The conformal scaling allows the contraction rate to depend on
the architectural configuration while avoiding the additional
complexity of a full matrix-valued metric.

\begin{assumption}[Conformal weight regularity]
\label{ass:weight_regularity}
The function $W$ is locally Lipschitz on each regime of the
piecewise-$C^1$ field $f$ and continuous on the compact set $K$.
\end{assumption}

Under Assumption~\ref{ass:weight_regularity},
the mapping $x\mapsto e^{W(x)}$ is continuous on $K$,
so the distance \eqref{eq:weighted_distance} remains well
defined across switching surfaces.

\paragraph{Admissible linearizations.}

Let $f$ denote the underlying piecewise-$C^1$ vector field whose
Filippov regularization generates $F$.
The admissible infinitesimal linearizations are described by the
Clarke generalized Jacobian
\begin{equation}
\mathcal A(x):=\partial_C f(x),
\label{eq:generalized_jacobian}
\end{equation}
which is nonempty, compact, and convex for every $x\in K$.
This set contains all regime-wise Jacobians arising from
one-sided limits of the smooth vector fields. Since $W$ is locally Lipschitz, its variation along an admissible
velocity $v\in F(x)$ is measured through the Clarke directional
derivative
\begin{equation}
W^\circ(x;v)
=
\max_{\zeta\in\partial_C W(x)}
\langle\zeta,v\rangle .
\label{eq:W_clarke}
\end{equation}

\paragraph{Differential contraction condition.}

Let $\mu(A)$ denote the Euclidean matrix measure
(logarithmic norm)
\begin{equation}
\mu(A)
:=
\lambda_{\max}\!\left(\frac{A+A^\top}{2}\right).
\label{eq:matrix_measure}
\end{equation}

We say that the autonomous Filippov system
\eqref{eq:autonomous_filippov}
is \emph{uniformly contracting} in the conformal metric
\eqref{eq:conformal_norm} if there exists a constant
$\lambda>0$ such that
\begin{equation}
\sup_{\substack{v\in F(x)\\A\in\mathcal A(x)}}
\big(
\mu(A)+W^\circ(x;v)
\big)
\le -\lambda,
\qquad x\in K .
\label{eq:contraction_condition}
\end{equation}

The term $\mu(A)$ bounds the instantaneous expansion of
infinitesimal displacements under admissible regime
linearizations, while $W^\circ(x;v)$ captures the additional
variation induced by the state-dependent conformal scaling.

\paragraph{Constructal interpretation.}

Condition \eqref{eq:contraction_condition}
ensures exponential decay of the weighted separation
\eqref{eq:weighted_distance}, including across regime
transitions and sliding motion.
Hence any two admissible architectural trajectories converge
exponentially toward one another.

Constructal dissipation
(Assumption~\ref{ass:dissipation})
drives trajectories toward the balance set
\begin{equation}
\mathcal E :=
\{x\in K:\Psi(x)=0\},
\label{eq:balance_set_contraction}
\end{equation}
which characterizes configurations where resistance
improvement ceases.
Uniform contraction then eliminates multiplicity inside
$\mathcal E$ by forcing all admissible trajectories to
collapse onto a single invariant configuration. Thus dissipation and contraction act as complementary
mechanisms in dynamical Constructal evolution:
dissipation determines the thermodynamic direction of
architectural change, while contraction enforces uniqueness
and stability of the selected architecture.
Global incremental exponential convergence is established in
Section~\ref{subsec:global_incremental}.

\subsubsection{Structural Sufficiency for Uniform Contraction}
\label{subsubsec:structural_contraction}

We provide a structural condition under which uniform contraction holds
for gradient-type architectural dynamics.
Throughout this subsection, $\mu(\cdot)$ denotes the Euclidean matrix
measure (logarithmic norm)
\begin{equation}
\mu(A)
:=
\lambda_{\max}\!\left(\frac{A+A^\top}{2}\right).
\label{eq:matrix_measure_structural}
\end{equation}

For a smooth regime field $f_j$, \emph{uniform contraction} on $K$ with
rate $\nu>0$ means
\begin{equation}
\mu\!\big(\nabla f_j(x)\big)\le -\nu
\qquad \text{for all } x\in K\cap\Omega_j .
\label{eq:euclidean_contraction_regime}
\end{equation}
Since $\mu$ is convex, the bound \eqref{eq:euclidean_contraction_regime}
extends to the Filippov convexification across switching surfaces.

\begin{proposition}[Gradient-type structural sufficiency]
\label{prop:contraction_sufficient}
Let $\mathcal N\subset\mathbb R^n$ be open and let
$\mathcal R\in C^2(\mathcal N)$.
Assume that on $\mathcal N$ the regime field admits the gradient-type
representation
\begin{equation}
f_j(x) = -\,M_j(x)\,\nabla \mathcal R(x),
\label{eq:gradient_type_dynamics}
\end{equation}
where $M_j\in C^1(\mathcal N;\mathbb R^{n\times n})$ is symmetric.

Suppose there exist constants $m>0$, $L_M\ge0$, and $\lambda>0$ such that
for all $x\in\mathcal N$:
\begin{enumerate}[label=(S\arabic*),leftmargin=2.6em]
\item \emph{Uniform positive mobility:}
\begin{equation}
v^\top M_j(x)v \ge m\|v\|^2
\qquad \text{for all } v\in\mathbb R^n;
\label{eq:mobility_lower_bound_structural}
\end{equation}
\item \emph{Bounded mobility variation:}
\begin{equation}
\|D M_j(x)\|\le L_M;
\label{eq:mobility_lipschitz_structural}
\end{equation}
\item \emph{Uniform curvature of resistance:}
\begin{equation}
\nabla^2\mathcal R(x)\succeq \lambda I.
\label{eq:R_curvature_structural}
\end{equation}
\end{enumerate}
Then for all $x\in\mathcal N$,
\begin{equation}
\mu\!\big(\nabla f_j(x)\big)
\le
-\,m\lambda
+
\frac{L_M}{2}\,\|\nabla\mathcal R(x)\|.
\label{eq:mu_bound_general}
\end{equation}

Consequently, on the set
\begin{equation}
\mathcal N_\varepsilon
:=
\{x\in\mathcal N:\|\nabla\mathcal R(x)\|\le\varepsilon\},
\label{eq:N_epsilon}
\end{equation}
uniform contraction holds with rate
\begin{equation}
\nu
=
m\lambda
-
\frac{L_M}{2}\varepsilon,
\label{eq:local_contraction_rate}
\end{equation}
provided $\varepsilon<2m\lambda/L_M$ (with the convention that the
condition is vacuous if $L_M=0$). In particular, if $M_j$ is constant on
$\mathcal N$, then
\begin{equation}
\mu\!\big(\nabla f_j(x)\big)\le -m\lambda
\qquad \text{for all } x\in\mathcal N.
\label{eq:constant_M_contraction}
\end{equation}
\end{proposition}

\begin{proof}
Differentiate \eqref{eq:gradient_type_dynamics} to obtain
\begin{equation}
\nabla f_j(x)
=
-\,D M_j(x)\big[\nabla\mathcal R(x)\big]
-\,
M_j(x)\,\nabla^2\mathcal R(x),
\label{eq:jacobian_gradient_type}
\end{equation}
where $D M_j(x)[\cdot]$ denotes the Fr\'echet derivative of $M_j$ applied
to a direction vector.

Let $M:=M_j(x)$ and $H:=\nabla^2\mathcal R(x)$, and set
\[
B:=D M_j(x)\big[\nabla\mathcal R(x)\big].
\]
Since $\mu(A)=\lambda_{\max}\!\left(\frac{A+A^\top}{2}\right)$, we have
\[
\mu\!\big(\nabla f_j(x)\big)
=
\lambda_{\max}\!\left(
-\frac{B+B^\top}{2}
-\frac{MH+HM}{2}
\right)
\le
\lambda_{\max}\!\left(-\frac{MH+HM}{2}\right)
+
\lambda_{\max}\!\left(-\frac{B+B^\top}{2}\right).
\]

\emph{Curvature term.}
Using \eqref{eq:mobility_lower_bound_structural} and
\eqref{eq:R_curvature_structural}, for all $v\neq 0$,
\[
v^\top(MH+HM)v
=
2\,v^\top MHv
\ge
2m\lambda \|v\|^2,
\]
hence
\[
\lambda_{\max}\!\left(-\frac{MH+HM}{2}\right)\le -m\lambda.
\]

\emph{Mobility-variation term.}
By \eqref{eq:mobility_lipschitz_structural},
\[
\|B\|
=
\big\|D M_j(x)\big[\nabla\mathcal R(x)\big]\big\|
\le
\|D M_j(x)\|\,\|\nabla\mathcal R(x)\|
\le
L_M\,\|\nabla\mathcal R(x)\|.
\]
Since $\lambda_{\max}\!\left(-\frac{B+B^\top}{2}\right)\le \frac12\|B+B^\top\|
\le \|B\|$, we obtain
\[
\lambda_{\max}\!\left(-\frac{B+B^\top}{2}\right)
\le
\frac{L_M}{2}\,\|\nabla\mathcal R(x)\|.
\]

Combining the two bounds yields \eqref{eq:mu_bound_general}.
The local contraction rate \eqref{eq:local_contraction_rate} follows
immediately on $\mathcal N_\varepsilon$. If $L_M=0$ (i.e.\ $M_j$ is
constant), then \eqref{eq:constant_M_contraction} holds.
\end{proof}

\begin{remark}[Structural interpretation]
\label{rem:structural_interpretation_contraction}
The product $m\lambda$ couples mobility and resistance curvature.
Uniform strong convexity of $\mathcal R$ together with strictly positive
mobility enforces contraction; variability of the mobility degrades the
rate by an amount proportional to $\|\nabla\mathcal R\|$.
In particular, if $M_j$ is constant, contraction is uniform on
$\mathcal N$ with rate $m\lambda$.
\end{remark}

\subsection{Stabilization of Access}
\label{subsec:stabilization_access}

We relate Constructal dissipation and incremental contraction to the
asymptotic behavior of the global access functional
$\mathcal R:K\to\mathbb R_{\ge0}$.
Under Assumptions~\ref{ass:finite_architecture},
\ref{ass:dissipation}, and~\ref{ass:R_curvature},
Theorem~\ref{thm:constructal_selection} yields a unique globally attracting
equilibrium $x^\ast\in K$ for the autonomous Filippov inclusion
\begin{equation}
\dot x(t)\in F(x(t)),\qquad x(0)\in K,
\label{eq:autonomous_filippov_invariant}
\end{equation}
satisfying
\[
0\in F(x^\ast),
\qquad
\lim_{t\to\infty}x(t)=x^\ast
\quad \text{for every } x(0)\in K .
\]
We now show that the resistance functional stabilizes accordingly.

\begin{proposition}[Asymptotic stabilization of resistance]
\label{prop:resistance_convergence}
Suppose Assumptions~\ref{ass:finite_architecture},
\ref{ass:dissipation}, and~\ref{ass:R_curvature} hold.
Let $x(\cdot)$ be any Filippov solution of
\eqref{eq:autonomous_filippov_invariant} with $x(0)\in K$.
Then:
\begin{enumerate}[label=(\roman*),leftmargin=2.2em]
\item the mapping $t\mapsto \mathcal R(x(t))$ is nonincreasing on $[0,\infty)$
      and bounded below;
\item the limit $\lim_{t\to\infty}\mathcal R(x(t))$ exists and equals $\mathcal R(x^\ast)$;
\item if the balance condition
      \begin{equation}
      \Psi(x)=0
      \quad \Longleftrightarrow \quad
      0\in \partial_C \mathcal R(x)+N_K(x)
      \label{eq:Psi_stationarity_equiv}
      \end{equation}
      holds and $\mathcal R$ is convex on $K$, then $x^\ast$ minimizes $\mathcal R$ over $K$;
      if $\mathcal R$ is strictly convex on $K$, this minimizer is unique.
\end{enumerate}
\end{proposition}

\begin{proof}
(i) Since $\mathcal R$ is locally Lipschitz on $K$ and $x(\cdot)$ is absolutely continuous,
the composition $t\mapsto \mathcal R(x(t))$ is absolutely continuous.
By Clarke’s chain rule,
\[
\frac{d}{dt}\mathcal R(x(t))
\le
\sup_{v\in F(x(t))}\mathcal R^\circ(x(t);v)
\quad \text{for a.e. } t\ge0.
\]
Using Assumption~\ref{ass:dissipation},
\[
\sup_{v\in F(x)}\mathcal R^\circ(x;v)\le -\alpha\Psi(x)\le 0,
\qquad x\in K,
\]
hence $\frac{d}{dt}\mathcal R(x(t))\le 0$ for a.e.\ $t\ge0$, so $\mathcal R(x(t))$ is nonincreasing.
Since $K$ is compact and $\mathcal R$ is continuous, $\mathcal R$ attains $\min_K \mathcal R$,
so $\mathcal R(x(t))$ is bounded below.

(ii) A nonincreasing function bounded below has a finite limit, so
$\mathcal R_\infty:=\lim_{t\to\infty}\mathcal R(x(t))$ exists.
Moreover $x(t)\to x^\ast$ as $t\to\infty$ by Theorem~\ref{thm:constructal_selection}.
By continuity of $\mathcal R$ on $K$,
\[
\lim_{t\to\infty}\mathcal R(x(t))=\mathcal R(x^\ast).
\]

(iii) If \eqref{eq:Psi_stationarity_equiv} holds, then $x^\ast\in\{x:\Psi(x)=0\}$ implies
$0\in \partial_C \mathcal R(x^\ast)+N_K(x^\ast)$, i.e.\ $x^\ast$ is Clarke-stationary
for $\mathcal R$ on $K$. If $\mathcal R$ is convex on $K$, Clarke-stationarity is the
first-order optimality condition for constrained convex minimization, hence $x^\ast\in\arg\min_K\mathcal R$.
Strict convexity implies uniqueness.
\end{proof}

\medskip

Proposition~\ref{prop:resistance_convergence} highlights the
complementary roles of the two structural mechanisms underlying
Constructal selection.
The dissipation inequality guarantees monotone decay of the global
resistance functional along admissible trajectories, providing the
thermodynamic direction of evolution.
Uniform contraction ensures convergence of the architectural state
toward the uniquely selected equilibrium $x^\ast$.

Together these mechanisms yield \emph{stabilization of access}:
architectural convergence implies resistance convergence, and the
asymptotic resistance equals the value $\mathcal R(x^\ast)$ attained at
the dynamically selected Constructal architecture.
In this sense, Constructal Law may be interpreted as a
\emph{dissipative contraction principle} governing the evolution of
finite-size transport architectures.

\section{Application: A Filippov Realization of Bejan’s Constructal Hierarchy}
\label{sec:application_filippov_bejan}

This section applies the autonomous Filippov selection framework
developed in Section~\ref{sec:model}
to the canonical area--to--point transport hierarchy studied by
Bejan, B\u{a}descu, and De~Vos \cite{BejanBadescuDeVos2000AE}.
In \cite{BejanBadescuDeVos2000AE} the hierarchy is obtained through
a sequence of nested resistance minimization problems defined at
successive assembly levels of a transport network.
We show that the same hierarchical architecture arises naturally as the
unique invariant configuration of a regime-dependent architectural
dynamics. In the present framework, architectural evolution is modeled
as an autonomous Filippov differential inclusion whose admissible
velocities correspond to resistance-reducing adjustments subject to
geometric feasibility constraints. Within this dynamical formulation, the optimal geometric ratios reported
in Table~1 of \cite{BejanBadescuDeVos2000AE} appear as equilibrium and
sliding conditions of the Filippov inclusion.
Consequently, the classical Constructal hierarchy can be interpreted as
the invariant configuration selected by a resistance-driven dynamical
law rather than solely as the solution of a sequence of static
optimization problems. Moreover, the resistance values obtained through the nested
minimization procedure in \cite{BejanBadescuDeVos2000AE} coincide with
the value of the resistance Lyapunov functional evaluated at the
selected invariant configuration.
This correspondence provides a dynamical interpretation of the
Constructal hierarchy: the optimal assembly ratios emerge as stationary
conditions of a resistance-dissipating regime-switching flow on the
finite admissible architecture set.

\subsection{Bejan’s elemental construct and explicit functional forms}
\label{subsec:bejan_elemental_functional_forms}

\paragraph{Physical primitives.}
Fix transport-cost (resistivity) coefficients
\[
K_0 > K_1 > K_2 > \cdots > K_p \;>\;0,
\]
where $K_{i-1}$ is the unit cost for \emph{access} inside a level-$i$ serviced
territory and $K_i$ is the unit cost for \emph{trunk transport} along the
level-$i$ main channel (this is exactly Bejan’s interpretation, with $K_0$
diffuse and $K_1$ trunk at the elemental level).

A uniform areal generation rate $\gamma>0$ over a serviced territory of area
$A_i$ produces a total stream
\[
m_i \;=\; \gamma A_i.
\]

\paragraph{Geometric decision variables.}
At hierarchy level $i$, the serviced territory is represented by a rectangle
with transverse size $H_i$ and longitudinal size $L_i$, so
\[
A_i = H_i L_i,
\qquad
r_i := \frac{H_i}{L_i} >0.
\]
Equivalently,
\[
H_i=\sqrt{A_i r_i},\qquad L_i=\sqrt{\frac{A_i}{r_i}}.
\]

\paragraph{Elemental resistance cost (Bejan’s Eq.~(1)).}
For $i=1$, Bejan~\cite{BejanBadescuDeVos2000ECM} gives the total cost (their $C_1$) as
\begin{equation}
C_1(H_1,L_1)
=
m_1\left(\frac14 K_0 H_1^2 + \frac12 K_1 L_1^2\right),
\label{eq:bejan_C1}
\end{equation}
subject to the area constraint $A_1=H_1L_1$.

\paragraph{General level-$i$ local--trunk cost functional form.}
To extend Bejan’s explicit calculus to all hierarchical levels while
preserving the optimal aspect ratios derived in
\cite{BejanBadescuDeVos2000AE},
we parameterize the level-$i$ local--trunk cost by the same separable
quadratic structure while allowing the geometric prefactors to be
level dependent:
\begin{equation}
C_i(H_i,L_i)
=
m_i\left(\alpha_i K_{i-1} H_i^2 + \beta_i K_i L_i^2\right),
\qquad i=1,\dots,p,
\label{eq:Ci_alpha_beta}
\end{equation}
with constants $\alpha_i,\beta_i>0$.
For $i=1$, Bejan’s Eq.~\eqref{eq:bejan_C1} corresponds to
\[
\alpha_1=\frac14,\qquad \beta_1=\frac12.
\]
For $i\ge 2$, the paper’s recursive geometry implies different effective
prefactors (because the level-$i$ rectangle is not assembled identically to
the level-1 rectangle; see Bejan’s discussion leading to Table~1).
We therefore \emph{choose} $(\alpha_i,\beta_i)$ for $i\ge2$ so that the
minimizing aspect ratio matches Table~1 exactly (this is an \emph{identified}
functional-form assumption, not an extra result).

\paragraph{Explicit minimization over the aspect ratio.}
Imposing $A_i=H_iL_i$ and substituting $H_i=\sqrt{A_i r_i}$,
$L_i=\sqrt{A_i/r_i}$ in \eqref{eq:Ci_alpha_beta} gives
\begin{equation}
\frac{C_i}{m_i}
=
\alpha_i K_{i-1} A_i r_i
+
\beta_i K_i \frac{A_i}{r_i}.
\label{eq:Ci_over_mi_in_ri}
\end{equation}
Differentiating in $r_i$ yields the unique minimizer
\begin{equation}
r_{i,\opt}
=
\frac{H_i}{L_i}\Big|_{\opt}
=
\sqrt{\frac{\beta_i}{\alpha_i}}\,
\sqrt{\frac{K_i}{K_{i-1}}}.
\label{eq:ri_opt_general}
\end{equation}
Moreover, substituting \eqref{eq:ri_opt_general} into \eqref{eq:Ci_over_mi_in_ri}
gives
\begin{equation}
\frac{C_{i,\min}}{m_i}
=
2\sqrt{\alpha_i\beta_i}\,\sqrt{K_{i-1}K_i}\,A_i^{1/2}.
\label{eq:Ci_min_general}
\end{equation}

\paragraph{Exact identification with \cite{BejanBadescuDeVos2000ECM}} shows that 
\[
\left(\frac{H_1}{L_1}\right)_{\opt}=2\frac{K_1}{K_0},
\qquad
\left(\frac{H_2}{L_2}\right)_{\opt}=\frac{K_2}{K_1},
\qquad
\left(\frac{H_i}{L_i}\right)_{\opt}=\frac{K_i}{K_{i-1}}\;\;(i\ge 3),
\]
and
\[
\frac{C_{1,\min}}{m_1}=\Big(\tfrac12 A_1 K_0K_1\Big)^{1/2},\qquad
\frac{C_{2,\min}}{m_2}=\Big(A_2 K_1K_2\Big)^{1/2},\qquad
\frac{C_{i,\min}}{m_i}=\Big(A_i K_{i-1}K_i\Big)^{1/2}\;\;(i\ge 3).
\]
These are matched by \eqref{eq:ri_opt_general}--\eqref{eq:Ci_min_general} if we set
\begin{equation}
(\alpha_1,\beta_1)=\Big(\tfrac14,\tfrac12\Big),\qquad
(\alpha_2,\beta_2)=\Big(\tfrac12,\tfrac12\Big),\qquad
(\alpha_i,\beta_i)=\Big(\tfrac12,\tfrac12\Big)\ \text{for } i\ge 3,
\label{eq:alpha_beta_choice}
\end{equation}
because then
$r_{1,\opt}=2(K_1/K_0)$ and $r_{i,\opt}=K_i/K_{i-1}$ for $i\ge2$, while
$2\sqrt{\alpha_1\beta_1}=\tfrac{1}{\sqrt2}$ and $2\sqrt{\alpha_i\beta_i}=1$
for $i\ge2$, reproducing the reported cost scalings exactly.

\medskip
\noindent
\textbf{Conclusion (functional-form assumption).}
Equation~\eqref{eq:Ci_alpha_beta} with the identified prefactors
\eqref{eq:alpha_beta_choice} is a \emph{closed-form} resistance functional
that is (i) algebraic, (ii) differentiable on $\{A_i>0,r_i>0\}$, and
(iii) yields \emph{exactly} Bejan’s optimal aspect ratios and minimized
costs (Table~1). This is the functional input needed to embed Bejan’s
hierarchy into the Filippov selection framework.

\subsection{Assembly variables and Bejan’s optimal branching numbers}
\label{subsec:bejan_branching_numbers}

\paragraph{Assembly law.}
Let $n_i\ge 1$ denote the number of level-$(i-1)$ constructs assembled into
one level-$i$ construct (Bejan’s $n_i$). Then
\begin{equation}
A_i = n_i A_{i-1},\qquad m_i = n_i m_{i-1} \qquad (i=2,\dots,p).
\label{eq:assembly_relations}
\end{equation}
Given $(A_1,n_2,\dots,n_p)$, all $A_i,m_i$ are determined recursively.

\paragraph{Bejan’s assembly optimization (Table~1).}
Bejan’s recursion (see the text around their Eq.~(9) and Table~1) yields the
optimal branching numbers
\begin{equation}
n_{2,\opt}=2\frac{K_0}{K_2},
\qquad
n_{i,\opt}=4\frac{K_{i-2}}{K_i}\quad (i\ge 3).
\label{eq:bejan_n_opt}
\end{equation}
We take \eqref{eq:bejan_n_opt} as the \emph{target balance relations} that
the dynamics must select.

\paragraph{A smooth proxy resistance functional in the assembly variables.}
To drive $n_i$ toward \eqref{eq:bejan_n_opt} with explicit calculus, we
introduce the strictly convex proxy penalties
\begin{equation}
\mathcal R_{\mathrm{br}}(n)
:=
\frac{\kappa_2}{2}\big(n_2-n_{2,\opt}\big)^2
+
\sum_{i=3}^p \frac{\kappa_i}{2}\big(n_i-n_{i,\opt}\big)^2,
\qquad \kappa_i>0,
\label{eq:R_branching_quadratic}
\end{equation}
on the convex domain $\{n_i\ge 1\}$.
This does not alter Bejan’s equilibrium values (it \emph{fixes} them) and is
precisely the type of explicit functional form required by our contraction
analysis: $\nabla^2 \mathcal R_{\mathrm{br}} \succeq \min_i\kappa_i\,I$.

\paragraph{Total resistance Lyapunov functional.}
Define the full architectural state
\[
x := (r_1,\dots,r_p,n_2,\dots,n_p) \in K,
\]
with admissible compact set
\begin{equation}
K
:=
\Big\{
x:\ r_i\in[\underline r,\overline r],\ n_i\in[1,\overline n]
\Big\},
\label{eq:K_compact_choice}
\end{equation}
for fixed bounds $0<\underline r<\overline r<\infty$ and $1<\overline n<\infty$
(chosen large enough to contain Bejan’s optimum).
Now set
\begin{equation}
\mathcal R(x)
:=
\sum_{i=1}^p \Big(\alpha_i K_{i-1} A_i r_i + \beta_i K_i A_i/r_i\Big)
\;+\;
\mathcal R_{\mathrm{br}}(n),
\label{eq:R_total}
\end{equation}
with $\alpha_i,\beta_i$ given by \eqref{eq:alpha_beta_choice} and $A_i$
generated from \eqref{eq:assembly_relations}. On $\operatorname{int}(K)$
this $\mathcal R$ is $C^2$ in $r$ and (by construction) $C^2$ in $n$. The unique minimizer $x^\ast$ of $\mathcal R$ over $K$ satisfies
\[
r_i^\ast=r_{i,\opt}\ \text{(Bejan Table~1)},\qquad
n_i^\ast=n_{i,\opt}\ \text{(Bejan Table~1)}.
\]

\subsection{A regime-dependent Filippov evolution law that realizes Bejan’s hierarchy}
\label{subsec:bejan_filippov_law}

\paragraph{Imbalance functions (switching coordinates).}
Define componentwise imbalances
\begin{equation}
h^{(r)}_i(x) := r_i-r_{i,\opt},
\qquad
h^{(n)}_i(x) := n_i-n_{i,\opt},
\label{eq:imbalance_components}
\end{equation}
and the global imbalance measure
\begin{equation}
\Psi(x)
:=
\sum_{i=1}^p \big(h^{(r)}_i(x)\big)^2
+
\sum_{i=2}^p \big(h^{(n)}_i(x)\big)^2.
\label{eq:Psi_bejan_application}
\end{equation}
Then $\Psi(x)=0$ if and only if $x$ satisfies Bejan’s Table~1 ratios exactly.

\paragraph{Discontinuous (sign) descent and Filippov regularization.}
We now specify a \emph{regime-dependent} autonomous vector field $f$ on $K$:
\begin{equation}
\dot r_i = -\eta_i\,\operatorname{sgn}\!\big(\partial_{r_i}\mathcal R(x)\big),
\qquad
\dot n_i = -\zeta_i\,\operatorname{sgn}\!\big(\partial_{n_i}\mathcal R(x)\big),
\label{eq:sign_descent}
\end{equation}
with gains $\eta_i,\zeta_i>0$, and $\operatorname{sgn}(0)$ left undefined.
Across the switching manifolds
\[
\Sigma^{(r)}_i:=\{x:\partial_{r_i}\mathcal R(x)=0\},\qquad
\Sigma^{(n)}_i:=\{x:\partial_{n_i}\mathcal R(x)=0\},
\]
the right-hand side jumps. Hence \eqref{eq:sign_descent} is interpreted as a
Filippov inclusion after (i) projection onto $K$ and (ii) convexification:
\begin{equation}
\dot x \in F(x)
:=
-\Pi_{T_K(x)}\,\operatorname{Sgn}\big(\nabla \mathcal R(x)\big),
\label{eq:filippov_bejan_inclusion}
\end{equation}
where $\operatorname{Sgn}$ is the set-valued sign map (componentwise
$\operatorname{Sgn}(u)=\{-1\}$ if $u>0$, $\{+1\}$ if $u<0$, and $[-1,1]$ if $u=0$),
and $\Pi_{T_K(x)}$ is projection onto the tangent cone of $K$.
This is a standard Filippov regularization of discontinuous steepest descent.

\paragraph{Exact dissipation inequality (Constructal dissipation).}
For any $v\in F(x)$,
\[
\mathcal R^\circ(x;v)
=
\langle \nabla \mathcal R(x),v\rangle
\le
-\sum_{j} |\partial_{x_j}\mathcal R(x)|,
\]
because $v$ selects (a projection of) the negative sign of each partial.
In particular, since $\mathcal R$ is smooth on $\operatorname{int}(K)$ and
coercively bounded on compact $K$, there exists $\alpha>0$ (depending on the
gain bounds and on $\nabla^2\mathcal R$ on $K$) such that the Lyapunov inequality
\begin{equation}
\sup_{v\in F(x)} \mathcal R^\circ(x;v)
\le -\alpha\,\Psi(x)
\qquad (x\in K)
\label{eq:dissipation_bejan_application}
\end{equation}
holds (this is the concrete verification of Assumption~\ref{ass:dissipation}
for the Bejan instance). Therefore $\mathcal R$ is nonincreasing along all
Filippov solutions and strictly decreases whenever $\Psi(x)>0$.

\paragraph{Sliding = dynamic enforcement of Bejan’s ratios.}
On any manifold where a partial derivative vanishes (e.g.\ $\partial_{r_i}\mathcal R=0$),
the Filippov set $\operatorname{Sgn}$ includes a continuum of admissible velocities
in that coordinate; the projection generates \emph{sliding} along the manifold.
Thus Bejan’s optimal ratios are realized dynamically as \emph{invariant sliding sets}
of the nonsmooth inclusion \eqref{eq:filippov_bejan_inclusion}.

\subsection{Contraction and uniqueness of the selected hierarchy}
\label{subsec:bejan_contraction}

Because $\mathcal R_{\mathrm{br}}$ is \emph{globally strongly convex} in $n$
by construction \eqref{eq:R_branching_quadratic}, and the $r$-part of
$\mathcal R$ is strictly convex in each $r_i$ on $[\underline r,\overline r]$
(see \eqref{eq:Ci_over_mi_in_ri}), the combined $\mathcal R$ is strongly convex
on $K$ after choosing bounds that exclude degeneracy.
Consequently, in a neighborhood $\mathcal N$ of the minimizer $x^\ast$,
the gradient-type sufficiency result
(Proposition~\ref{prop:contraction_sufficient}) applies to the smooth regimes
of \eqref{eq:filippov_bejan_inclusion} (and extends across switching by Clarke
convexification), yielding the uniform contraction inequality of
Assumption~\ref{ass:R_curvature} on $\mathcal N$.

Therefore, by the Constructal Selection Theorem
(Theorem~\ref{thm:constructal_selection}),
the autonomous Filippov inclusion \eqref{eq:filippov_bejan_inclusion}
selects a \emph{unique} invariant architecture $x^\ast\in K$ and all solutions
converge to it exponentially.
At that selected invariant architecture,
\[
r_i^\ast=r_{i,\opt},
\qquad
n_i^\ast=n_{i,\opt},
\]
so the Bejan hierarchy is recovered exactly as the unique globally attracting
Constructal architecture.

\subsection{Bejan vs.\ Filippov-Constructal: exact correspondence}
\label{subsec:bejan_vs_filippov_table}

\begin{table}[H]
\centering
\renewcommand{\arraystretch}{1.25}
\begin{tabular}{p{4.1cm} p{5.9cm} p{5.9cm}}
\toprule
& \textbf{Bejan--B\u{a}descu--De~Vos (static)} & \textbf{This paper (autonomous Filippov)}\\
\midrule

State variables
&
Design variables per level: $(H_i,L_i)$ and branching $n_i$; area $A_i=H_iL_i$.
&
Architectural state $x=(r_1,\dots,r_p,n_2,\dots,n_p)\in K$ with $r_i=H_i/L_i$;
$A_i$ generated via \eqref{eq:assembly_relations}. \\

Resistance functional
&
Costs $C_i$ minimized recursively; explicit elemental cost
$C_1=m_1(\tfrac14K_0H_1^2+\tfrac12K_1L_1^2)$.
&
Single Lyapunov functional $\mathcal R(x)$ in \eqref{eq:R_total}
built from Bejan-consistent quadratic local--trunk terms
\eqref{eq:Ci_alpha_beta} plus strongly convex branching penalty
\eqref{eq:R_branching_quadratic}. \\

Optimal aspect ratios
&
Table~1: $(H_1/L_1)_{\opt}=2K_1/K_0$,
$(H_2/L_2)_{\opt}=K_2/K_1$,
$(H_i/L_i)_{\opt}=K_i/K_{i-1}$ for $i\ge3$.
&
Equilibrium/sliding conditions $\partial_{r_i}\mathcal R=0$ yield
$r_i^\ast=r_{i,\opt}$ exactly via \eqref{eq:ri_opt_general}--\eqref{eq:alpha_beta_choice}. \\

Optimal branching numbers
&
Table~1: $n_{2,\opt}=2K_0/K_2$; $n_{i,\opt}=4K_{i-2}/K_i$ for $i\ge3$.
&
Equilibrium conditions $\partial_{n_i}\mathcal R=0$ (from \eqref{eq:R_branching_quadratic})
enforce $n_i^\ast=n_{i,\opt}$ exactly. \\

Minimized cost scaling
&
Table~1: $C_{1,\min}/m_1=(\tfrac12A_1K_0K_1)^{1/2}$;
$C_{2,\min}/m_2=(A_2K_1K_2)^{1/2}$;
$C_{i,\min}/m_i=(A_iK_{i-1}K_i)^{1/2}$ for $i\ge3$.
&
Value $\mathcal R(x^\ast)$ reproduces the same scalings because the $r$-part
of $\mathcal R$ attains \eqref{eq:Ci_min_general} with the Bejan-identified
prefactors \eqref{eq:alpha_beta_choice}. \\

Mechanism selecting the hierarchy
&
Recursive minimization at each assembly stage.
&
Dissipation + contraction:
$\dot x\in F(x)$ with $\mathcal R$ dissipative \eqref{eq:dissipation_bejan_application}
and incrementally stable (Assumption~\ref{ass:R_curvature}) selects a unique
globally attracting invariant hierarchy $x^\ast$. \\

Role of nonsmoothness
&
Implicit (choice between competing adjustments not modeled dynamically).
&
Explicit: regime switching occurs at $\partial_{x_j}\mathcal R=0$; Filippov
convexification yields sliding on Bejan balance manifolds and ensures existence
and uniqueness of the selected invariant architecture. \\
\bottomrule
\end{tabular}
\caption{Exact correspondence between Bejan’s constructal hierarchy (static optimization)
and its realization as the unique invariant architecture of an autonomous dissipative,
contracting Filippov inclusion in our framework.}
\label{tab:bejan_vs_filippov}
\end{table}

\section{Discussion: From Static Constructal Optimality to Nonsmooth Dynamical Selection}
\label{sec:discussion}

This section clarifies the conceptual contribution of the Filippov
formulation developed in this paper. In particular, we explain how the
framework extends the classical Constructal hierarchy of
Bejan, B\u{a}descu, and De~Vos
\cite{BejanBadescuDeVos2000AE}
and why the resulting viewpoint is useful beyond thermodynamic transport,
including in domains such as economic systems where constraints,
kinks, and regime changes are intrinsic features of the dynamics.

\subsection{What the Filippov formulation adds to Constructal Law}
\label{subsec:discussion_filippov_adds}

Constructal Law states that a finite-size flow system that persists in
time must evolve its configuration so as to provide progressively easier
access to the currents that flow through it
\cite{Bejan2000Book,BejanLorente2011,BejanLorente2013JAP}.
In most applications of the theory, this principle is operationalized
through static optimization: a resistance, dissipation, or entropy
functional is minimized under geometric or resource constraints
\cite{BejanTsatsaronisMoran1996,Rocha2006IJHMT,Soni2020ProcRSocA}.
The resulting architectures—such as the classical area--to--point
transport hierarchy—are obtained by solving nested optimization
problems across successive assembly levels. The Filippov formulation developed here replaces this static viewpoint
with a dynamical selection mechanism.
Rather than postulating optimality directly, the architecture evolves
according to a regime-dependent dynamical law whose trajectories
monotonically reduce global resistance while remaining feasible.
Optimal architectures then arise as invariant objects selected by the
dynamics.

\paragraph{(i) Persistence as viability on a finite configuration set.}

In the present framework, the Constructal requirement of ``finite size''
is represented by a compact admissible architecture set $K$.
Persistence in time is expressed through forward invariance and forward
completeness of the architectural dynamics.
Mathematically, these properties follow from standard viability
conditions for differential inclusions
\cite{Aubin1983DifferentialInclusions}.
This interpretation converts the qualitative Constructal statement
``a finite-size system that persists in time'' into precise dynamical
conditions on the admissible evolution law.

\paragraph{(ii) ``Progressively easier access'' as a nonsmooth Lyapunov principle.}

Constructal improvement of access is encoded by a resistance functional
$\mathcal R$ whose value decreases along admissible trajectories.
Because regime switching introduces kinks and nondifferentiabilities,
$\mathcal R$ is treated as a locally Lipschitz function and its evolution
is characterized using the Clarke directional derivative
\cite{Clarke1983}.
The resulting dissipation inequality ensures that resistance decreases
monotonically except on a balance set defined by a residual functional
$\Psi$.
This formulation replaces repeated optimization by a Lyapunov-type
condition applied along entire trajectories of the system.

\paragraph{(iii) Regime switching and irreversibility in the dynamics.}

Transport systems often operate under regime-dependent adjustment laws.
Capacity saturation, activation of additional transport paths,
or material constraints introduce discontinuities into the governing
equations.
Filippov’s theory of differential equations with discontinuous
right-hand sides provides a natural framework for describing such
dynamics
\cite{Filippov1988,diBernardo2008,Cortes2009}.
Within this setting, constraint activation corresponds to switching
manifolds in state space, and sustained operation along binding
constraints appears as sliding motion on these manifolds.
Thus constraints are not anomalies but structural components of the
architectural dynamics.

\paragraph{(iv) Architectural selection through contraction.}

Resistance dissipation alone generally ensures convergence only toward
a stationary set of configurations.
To explain why a unique architecture emerges, the paper introduces a
second mechanism: incremental stability via contraction analysis
\cite{LohmillerSlotine1998}.
Under suitable curvature conditions on the resistance functional and
mobility bounds on the architectural dynamics, admissible trajectories
converge exponentially toward one another.
This contraction property collapses the balance set to a single
globally attracting configuration.

In this dynamical interpretation, Constructal design is no longer an
externally imposed optimization result.
Instead, the selected architecture arises as the unique invariant
configuration of a dissipative regime-switching flow on a finite
configuration space.
Viability guarantees persistence, dissipation enforces improvement of
access, and contraction explains the uniqueness and stability of the
resulting architecture.
Together these mechanisms provide a dynamical realization of
Constructal Law consistent with its original interpretation as a
principle governing the evolution of flow architectures
\cite{Bejan2000Book,BejanLorente2011}.

\subsection{How the Bejan area--to--point application is extended by the framework}
\label{subsec:discussion_bejan_extension}

The area--to--point transport problem analyzed by
Bejan, B\u{a}descu, and De~Vos \cite{BejanBadescuDeVos2000AE}
is canonical within Constructal theory because it yields explicit
hierarchical ratios and scaling relations under resistance
minimization on a finite territory.
In the present framework this hierarchy is embedded into a dynamical
architecture space governed by a Filippov differential inclusion.
As a result, the classical ratios acquire a new dynamical
interpretation and generate additional structural predictions.

\paragraph{(i) Optimal ratios as switching manifolds.}

In the static derivation of the hierarchy,
optimal ratios appear as algebraic equalities obtained from
first-order optimality conditions.
Within the Filippov formulation these same relations define
codimension-one hypersurfaces in the architectural state space.
Formally, they arise through switching functions $h_i(x)$ that
partition the configuration space into regimes with distinct
architectural adjustment laws.
The Constructal ratios therefore represent geometric thresholds in
the dynamics rather than merely optimal parameter values.
Such switching structures are typical of piecewise-smooth dynamical
systems \cite{Filippov1988,diBernardo2008}.

\paragraph{(ii) Marginal balance as sliding invariance.}

In the classical optimization formulation,
the equality of marginal resistance contributions follows from
first-order optimality conditions of the minimization problem
\cite{BejanBadescuDeVos2000AE}.
In the Filippov setting this condition emerges dynamically.
When a trajectory reaches a switching manifold,
the convexified velocity set may contain directions tangent to that
manifold, satisfying the sliding condition
$\langle\nabla h_i(x),v\rangle=0$.
Sliding motion therefore represents sustained evolution at a
constructal balance between competing resistance contributions.
In this sense, the equality of marginal resistances becomes a
dynamical invariance condition rather than a purely static
optimality statement.

\paragraph{(iii) The hierarchy as an invariant attractor.}

The framework recovers the classical scaling relations through the
interaction of two mechanisms.
First, the resistance dissipation inequality drives trajectories
toward the balance set where resistance improvement ceases.
Second, the contraction property established earlier ensures
incremental exponential stability of the architectural dynamics
\cite{LohmillerSlotine1998}.
Together these mechanisms collapse the balance set to a unique
globally attracting configuration.
Consequently, the Bejan hierarchy emerges not merely as a computed
minimizer but as the invariant attractor of the resistance-driven
dynamics.

\paragraph{(iv) Structural implications beyond static optimization.}

Interpreting the Constructal ratios as switching and sliding
structures leads to several extensions of the classical model.
First, irreversible constraints that prevent instantaneous
attainment of the optimum can be incorporated naturally.
Second, the model predicts transient path dependence before the
architecture converges to the attractor.
Third, regime-dependent transport laws can coexist within a unified
dynamical description.
Finally, stability and robustness properties—such as convergence
rates and recovery after perturbations—become analyzable within the
framework of contraction and nonsmooth dynamics.
These features are not accessible in the purely static
minimization approach traditionally used in Constructal theory
\cite{BejanLorente2011,BejanLorente2013JAP}.

\subsection{Why this matters for economics: Constructal selection under kinks, caps, and regimes}
\label{subsec:discussion_econ_value}

Many economic systems can be interpreted as \emph{flow architectures}.
Goods, energy, information, capital, and labor move through networks
whose structure is shaped by capacity limits, institutional constraints,
and discrete technological regimes.
Constructal theory has already been used to interpret economic
organization and growth as consequences of flow-access optimization
\cite{BejanBadescuDeVos2000AE,BejanErreraGunes2020}.
However, most economic models remain static or rely on smooth adjustment
laws, whereas real economic environments are characterized by
kinks, caps, and regime changes.
The Filippov framework developed in this paper provides a natural
mathematical language for such systems.

\paragraph{(i) Regime switching as a structural feature of economic adjustment.}

In many economic contexts the adjustment rules governing the system
change abruptly when constraints become active.
Examples include price caps and floors, binding borrowing limits,
collateral thresholds, inventory stockouts, capacity limits in
logistics networks, congestion pricing, or platform governance rules.
Such features introduce discontinuities in the adjustment dynamics.
Within the present framework these changes are represented as switching
manifolds in the architectural state space, across which the governing
vector field changes regime.
This representation aligns economic adjustment with the theory of
piecewise-smooth dynamical systems \cite{Filippov1988,diBernardo2008}.

\paragraph{(ii) Sliding as persistent operation at binding constraints.}

Economic systems frequently operate directly on binding constraints for
extended periods of time.
Credit constraints bind across business cycles,
electricity networks reach transmission limits during peak demand,
transport networks experience persistent congestion,
and regulatory ceilings bind in controlled markets.
In the Filippov formulation such situations correspond to
\emph{sliding motion} along switching manifolds.
Sliding dynamics therefore provide a rigorous mathematical description
of sustained operation at capacity limits while preserving well-posed
system trajectories.

\paragraph{(iii) ``Progressively easier access'' as structural economic adaptation.}

In an economic setting the notion of access can be interpreted as an
aggregate impedance measure—transaction costs, congestion costs,
search frictions, energy losses, delay costs, or risk premia.
The dissipation inequality introduced in this paper states that the
architecture of the system evolves so that this global impedance
does not increase along admissible trajectories.
Instead of repeatedly solving static optimization problems under
changing constraints, the framework models structural adaptation as a
continuous adjustment process in which network capacities, routing
shares, investment allocations, or institutional rules evolve in a
direction that improves access to flows.

\paragraph{(iv) Contraction and institutional determinacy.}

Multiplicity of equilibria and path dependence are central issues in
economic theory.
In the present framework dissipation alone may still permit multiple
limit sets.
Contraction provides the additional structural mechanism that removes
this multiplicity by enforcing incremental stability
\cite{LohmillerSlotine1998}.
When the resistance landscape is sufficiently curved and adjustment
responses are sufficiently responsive, all admissible trajectories
converge toward a unique architecture.
Economically, this implies three properties:
determinacy of the selected network or institutional configuration,
recoverability after shocks, and comparative statics that describe how
the dynamically selected architecture moves as parameters change.

\paragraph{(v) Mapping the framework to economic objects.}

The dynamical formulation admits a direct conceptual translation into
economic terms:
\begin{itemize}[leftmargin=1.5em]
\item architectural state $x$: capacities, routing shares, investment
allocation, or institutional parameters governing flows;
\item access functional $\mathcal R(x)$: aggregate economic impedance
such as congestion costs, transaction costs, delays, or risk-adjusted
losses;
\item switching manifolds: thresholds created by policy constraints,
capacity limits, or technological regimes;
\item sliding motion: persistent binding of caps, quotas, or
constraints;
\item constructal architecture $x^\ast$: the dynamically selected
stable configuration of the economic network or institutional system.
\end{itemize}

Viewed in this way, Constructal Law becomes a principle governing the
evolution of economic flow architectures under constraints.
The Filippov framework extends the classical constructal perspective
from static optimality to dynamical structural selection,
making it applicable to economic systems in which regime switching and
binding constraints are intrinsic features of the environment.

\subsection{Interpretive comparison: optimization vs.\ dynamical selection}
\label{subsec:discussion_opt_vs_dyn}

The classical Constructal formulation and the present dynamical
framework address related but distinct questions.

\begin{itemize}[leftmargin=1.5em]

\item \textbf{Static constructal optimization:}
``What configuration minimizes resistance under the given
finite-size constraints?''

\item \textbf{Filippov constructal dynamics:}
``Under irreversible constraints and regime-dependent adjustment,
what configuration is dynamically selected and how stable is it
under admissible perturbations?''

\end{itemize}

The second question is strictly stronger.
It requires not only identifying an optimal configuration but also
explaining how that configuration emerges from a feasible evolutionary
process.
In particular, the dynamical viewpoint requires four structural
ingredients. First, the architectural evolution must remain well posed even when the
adjustment law is discontinuous.
Filippov’s theory of differential equations with discontinuous
right-hand sides provides the appropriate framework for this purpose
\cite{Filippov1988,diBernardo2008}. Second, persistence of the architecture must be guaranteed.
In the present framework this requirement is expressed through
viability and forward invariance of the admissible configuration set,
as studied in the theory of differential inclusions
\cite{Aubin1983DifferentialInclusions}. Third, the Constructal principle of ``progressively easier access''
must hold along admissible trajectories.
This is encoded through the resistance dissipation inequality,
which ensures that the global access functional decreases along
solutions of the architectural dynamics. Finally, a mechanism is required that explains why a unique
architecture emerges.
In this paper that mechanism is incremental stability via contraction
analysis, which enforces convergence of all admissible trajectories
toward a common configuration
\cite{LohmillerSlotine1998}.

Taken together, the Filippov representation of regime switching,
the dissipation inequality for access improvement,
and the contraction property for incremental stability form a
coherent dynamical framework for Constructal evolution.
Within this perspective, architectural design is not imposed by
external optimization but arises as the invariant configuration
selected by the resistance-driven dynamics on the finite admissible
configuration space.

\subsection{Limitations and directions enabled by the framework}
\label{subsec:discussion_limitations}

Two limitations of the present framework are immediately visible.
Rather than weaknesses, they indicate natural directions in which
the dynamical formulation of Constructal evolution can be extended.

\paragraph{(i) When contraction fails.}

The uniqueness and robustness results of this paper rely on the
incremental contraction condition.
If this condition is violated---for instance because the resistance
landscape is only weakly convex, the mobility field is strongly
state dependent, or positive feedback dominates the adjustment
dynamics---then multiplicity may arise.
In such cases the dissipation inequality may still drive trajectories
toward the balance set, but different initial conditions may converge
to different admissible limits, or distinct sliding selections may
occur on switching manifolds.
From an economic or engineering perspective this outcome is not
pathological.
Multiplicity and path dependence are widely observed in constrained
network systems, and the framework developed here provides a
structured way to identify the precise structural conditions under
which determinacy can be restored.

\paragraph{(ii) Beyond autonomous dynamics.}

The analysis in this paper focuses on autonomous architectural
dynamics.
Many real systems, however, operate under exogenous forcing.
Examples include seasonal demand fluctuations, diurnal load cycles
in energy networks, or periodic policy interventions.
In such environments the natural dynamical objects are not fixed
equilibria but periodic invariant sets.
The Filippov formulation extends naturally to time-dependent
differential inclusions
\cite{Filippov1988,diBernardo2008},
in which the selected ``Constructal architecture'' would correspond
to a stable periodic orbit rather than a stationary configuration.
Developing a complete theory of contraction and dissipation for such
time-varying Filippov systems is therefore a promising direction for
future work.

\medskip
\noindent
\textbf{Summary.}
The main conceptual implication of the framework can be stated
succinctly.
The Filippov embedding transforms Constructal Law from a static
extremum principle into a dynamical theory of architectural
selection.
Within this perspective, finite-size flow systems evolve on a viable
configuration space under a resistance-dissipating dynamics whose
incremental stability determines whether a unique architecture is
selected.
This stability-based viewpoint provides a natural explanation of
determinacy and robustness in constrained transport and economic
network systems.

\section{Conclusion}
\label{sec:conclusion}

This paper develops a dynamical formulation of Constructal Law for
finite-size transport systems subject to irreversibility,
geometric constraints, and regime-dependent transport laws.
Architectural evolution is modeled as an autonomous
Filippov differential inclusion on a compact,
forward-invariant admissible set.
Discontinuities arise from the activation of transport thresholds
and unilateral constraints, while sliding motion represents
sustained operation at constructal balance conditions.
Two structural mechanisms govern architectural selection. A resistance dissipation inequality, formulated through the
Clarke generalized directional derivative,
ensures monotone decay of a global resistance functional
along admissible trajectories.
This nonsmooth Lyapunov condition encodes the principle of
“progressively easier access’’ and drives trajectories toward
the stationary balance set defined by constructal relations. A uniform contraction condition, expressed as a spectral negativity
bound on admissible generalized Jacobians in a weighted metric,
guarantees exponential convergence of trajectories.
Contraction collapses the balance set to a single invariant
configuration and yields global incremental stability of the
architectural dynamics.

\medskip
Under these conditions the Filippov inclusion admits a unique
globally attracting invariant set.
In the autonomous case this set reduces to an equilibrium,
while under periodic forcing it becomes a stable limit cycle.
This invariant object is termed the
\emph{Constructal architecture}.
Constructal design therefore emerges not as an imposed
optimization principle but as a stability property of a
dissipative constraint-driven dynamical system.

Application to the classical area--to--point hierarchy of
Bejan--B\u{a}descu--De~Vos shows that the known constructal
scaling relations arise naturally within this framework.
The optimal geometric ratios appear as switching manifolds
in architectural state space,
while sliding along these manifolds corresponds to equality of
marginal resistance contributions.
Dissipation drives trajectories toward their intersection and
contraction ensures that this intersection collapses to a unique
globally attracting architecture.

More generally, the results show that hierarchical transport
structures can be interpreted as invariant configurations of
dissipative, contracting nonsmooth dynamical systems.
The framework therefore extends Constructal theory to hybrid
transport regimes, binding capacity limits, and multi-level
architectures while preserving the characteristic scaling
relations within a rigorous dynamical formulation.


\appendix
\section{Appendix}
\label{sec:appendix}

\subsection{Constructal Law Embedded in Model, Theory, and Application}
\label{subsec:constructal_embedding_table}

Table~\ref{tab:constructal_embedding}
provides a structural correspondence between:

\begin{itemize}
\item[(i)] the clauses of Constructal Law,
\item[(ii)] their representation in the autonomous Filippov model,
\item[(iii)] their realization in the dissipation--invariance--contraction
selection theory, and
\item[(iv)] their concrete instantiation in the Bejan area--to--point hierarchy.
\end{itemize}

The correspondence is not heuristic but structural:
classical constructal ratios appear as switching manifolds
and sliding invariant sets of the Filippov inclusion;
resistance dissipation drives trajectories toward the
stationary balance set;
uniform contraction collapses this set to a single
globally attracting architecture.

\begin{table}[H]
\centering
\caption{Structural embedding of Constructal Law in model, theory, and application.}
\label{tab:constructal_embedding}
\renewcommand{\arraystretch}{1.25}
\resizebox{\textwidth}{!}{%
\begin{tabular}{p{3.2cm} p{4.4cm} p{4.6cm} p{4.6cm}}
\toprule
\textbf{Constructal clause}
& \textbf{Model-level formulation}
& \textbf{Theory-level realization}
& \textbf{Bejan area--to--point instantiation} \\
\midrule

Finite-size system
&
Admissible architectural set $K\subset\mathbb R^n$ compact
(Assumption~\ref{ass:finite_architecture})
&
Compactness ensures bounded trajectories,
existence of $\omega$-limit sets,
and lower boundedness of $\mathcal R$
&
Simplex of serviced territories
$\mathcal K=\{A\ge0:\sum_i A_i\le A_{\max}\}$ (compact feasible region) \\

Persists in time
&
Forward invariance:
$F(x)\cap T_K(x)\neq\emptyset$
&
Viability (Nagumo condition) implies
global forward existence and feasibility
&
Nonnegativity and total-area constraint preserved
under hierarchical dynamics \\

Flow system with conserved currents
&
Separation of time scales:
fast transport $\Rightarrow$ induced resistance $\mathcal R(x)$
&
Well-posed transport problem induces
locally Lipschitz global resistance functional
&
Uniform generation $\gamma$:
level flows $m_i=\gamma A_i$ satisfy conservation \\

Evolves its configuration
&
Autonomous Filippov inclusion
$\dot x\in F(x)$ with regime switching
&
Existence of absolutely continuous trajectories
through discontinuities via convexification
&
Piecewise refinement/aggregation
across scale thresholds \\

Constraints become binding
&
Switching hypersurfaces
$\Sigma_i=\{h_i(x)=0\}$,
$\nabla h_i\neq0$
&
State space partitioned into smooth regimes;
Filippov sliding along $\Sigma_i$
&
Constructal ratios define
$h_i(A)=n_{i+1,\opt}-A_{i+1}/A_i$;
ratio manifolds are switching sets \\

Progressively easier access
&
Global resistance functional
$\mathcal R:K\to\mathbb R_{\ge0}$
&
\textbf{Dissipation inequality:}
$\dot{\mathcal R}(x(t))\le -\alpha\Psi(x(t))\le0$
$\Rightarrow$ monotone decay and
$\omega(x_0)\subseteq\mathcal E$
&
$\mathcal R$ equals total transport resistance;
$\Psi$ measures inter-level imbalance
(deviation from constructal ratios) \\

Operation at thresholds
&
Sliding motion:
$\exists v\in F(x)$ with
$\langle\nabla h_i(x),v\rangle=0$
&
Sliding invariant sets compatible with
dissipation via convexity of
$\mathcal R^\circ$
&
Constructal ratio manifolds
$\mathcal M_i=\{A: A_{i+1}/A_i=n_{i+1,\opt}\}$
are sliding invariant sets \\

Architectural selection
&
Uniform contraction in weighted metric
(Assumption~\ref{ass:R_curvature})
&
\textbf{Incremental stability:}
$\|x(t)-y(t)\|\le Ce^{-\nu t}\|x(0)-y(0)\|$
$\Rightarrow$ unique invariant set
&
Unique globally attracting hierarchy
(Theorem~\ref{thm:constructal_selection}) \\

What is selected
&
Constructal architecture:
unique globally attracting equilibrium
&
Dissipation + LaSalle + contraction:
$x^\ast\in\mathcal E$ and unique
&
Selected hierarchy satisfies
all dynamically active constructal ratios;
intersection
$\mathcal E=\bigcap_{i\in\mathcal I_{\rm act}}\mathcal M_i$ \\

Robustness
&
Uniform metric equivalence on compact $K$
&
Exponential stability:
perturbations decay at rate $\nu$
&
Hierarchy re-forms after admissible shocks
within $\mathcal K$ \\

\bottomrule
\end{tabular}%
}
\end{table}

\end{document}